\documentclass[reqno, 11pt]{amsart}

\usepackage{amsmath,amssymb,amsthm,amsfonts, amscd, url}
\usepackage{appendix}
\usepackage[numeric]{amsrefs}
\input xy
\xyoption{all}

\usepackage[dvipsnames]{xcolor}
\usepackage{graphics}
\usepackage{graphicx}

\DeclareMathOperator{\gen}{gen}
\DeclareMathOperator{\Av}{Av}

\DeclareMathOperator{\Hom}{Hom}
\DeclareMathOperator{\Aut}{Aut}
\DeclareMathOperator{\Supp}{Supp}

\DeclareMathOperator{\iw}{Iw}

\newcommand{\ve}{\mathbf{v}}

\newtheorem*{de}{Definition}
\newtheorem*{nthm}{Theorem}
\newtheorem*{nprop}{Proposition}
\newtheorem*{nlem}{Lemma}

\newtheorem*{nclaim}{Claim}

\newcommand{\be}[1]{\begin{eqnarray} \label{#1}}

\newcommand{\ee}{\end{eqnarray}}

\newcommand{\tpoint}[1]{\vspace{3mm}\par \noindent \refstepcounter{subsection}{\bf \thesubsection.}
  {\em #1.} }

\newcommand{\rpoint}[1]{\vspace{3mm}\par \noindent{\em {{#1.} }}}
\newcommand{\spoint}{\vspace{3mm}\par \noindent \refstepcounter{subsection}{\bf \thesubsection.} }

\numberwithin{equation}{section}

\newcommand{\xv}{\xi^{\vee}}

\newcommand{\av}{a^{\vee}}

\newcommand{\aw}{\mathbb{W}}
\newcommand{\dd}{\mathbf{d}}
\newcommand{\cc}{\mathbf{c}}

\newcommand{\up}{U}
\newcommand{\um}{U^-}

\newcommand{\umw}{U^{-, w}}

\renewcommand{\v}{\mathbf{v}}
\newcommand{\la}{\langle}
\newcommand{\ra}{\rangle}
\newcommand{\lv}{\Lambda^{\vee}}
\newcommand{\mv}{\mu^{\vee}}
\newcommand{\res}{\kappa}
\newcommand{\mf}[1]{\mathfrak{#1}}
\newcommand{\f}[1]{\mathfrak{#1}}
\newcommand{\rr}{\rightarrow}
\newcommand{\mc}[1]{\mathcal{#1}}
\newcommand{\wt}[1]{\widetilde{#1}}

\newcommand{\Lv}{\Lambda^{\vee}}
\renewcommand{\lv}{\lambda^{\vee}}

\newcommand{\K}{\mathcal{K}}
\newcommand{\zee} {\mathbb{Z}}
\newcommand{\C} {\mathbb{C}}

\renewcommand{\O}{\mathcal{O}}
\newcommand{\qn}{{\mathcal Q}_{v}}
\newcommand{\qnf}{{\mathcal Q}^{fin}_{v}}

\newcommand{\niw}{\textbf{n}^-}
\newcommand{\aiw}{\textbf{a}}
\renewcommand{\c}{\mathbf{c}}
\renewcommand{\b}{\mathbf{b}}

\newcommand{\wid}{\mathbf{1}_{\psi}} 

\newcommand{\me}[1]{\mc{M}_o(#1)}

\newcommand{\W}{\mathcal{W}}

\newcommand{\kk}{\kappa}
\newcommand{\nm}{n^{-}}

\begin{document}

\title{Unramified Whittaker Functions on $p$-adic Loop Groups}
\author{Manish M. Patnaik} 
\email{patnaik@ualberta.ca}

\maketitle

\begin{abstract} We define unramified Whittaker functions on the $p$-adic points of an affine Kac-Moody group, and establish an analogue of the Casselman-Shalika formula for these functions.
\end{abstract}

\setcounter{tocdepth}{1}
\tableofcontents

\section{Introduction}

\spoint Recently, convolution Hecke algebras on spaces of double cosets for $p$-adic loop groups have been defined (\cite{kap, bk, bkp, gau:rou}), and explicit formulas for certain special functions in this theory were obtained: the affine spherical function was computed in \cite{bkp}, and the affine Gindikin-Karpelevic has been computed in two different ways-- first in \cite{bfk} using a cohomological interpretation which is only available in positive characteristic, and then in \cite{bgkp} using a different group theoretical argument which works in all in characteristics. In this note we turn our attention to another construction within the same venue, the unramified Whittaker function. We prove the following explicit formula for these functions generalizing the work of Casselman and Shalika \cite{cs} for finite-dimensional groups (see Theorem \ref{cs-fin} for a reminder of their result in our notation). 

\begin{nthm} [see Theorem \ref{main}] Let $G$ be an untwisted affine, simple, simply-laced Kac-Moody group over a nonarchimedean local field whose underlying finite-dimensional root system is simply-laced of rank $\ell,$ and let $G^{\vee}$ be the (Langlands) dual group to $G$. Let $\lv$ be a dominant affine coweight of $G$ and $\pi^{\lv}$ the corresponding element in a split torus of $G. $ The affine Whittaker function $\W$ (see Definition \ref{whitfn}) takes the following values, \be{cs-intro} \W(\pi^{\lv}) = q^{  \la \rho, \lv \ra } \mf{m} \prod_{\av \in R^{\vee}_+} (1 - q^{-1} e^{-\av }) \chi_{\lv}, \ee where $q$ is the size of the residue field, $\rho$ is the co-character of $G^{\vee}$ whose value on each simple root of $G^{\vee}$ is 1, $R^{\vee}_+$ is the set of all positive roots of $G^{\vee}$,  $\chi_{\lv}$ is the Weyl-Kac character of the irreducible highest weight-representation of $G^{\vee}$ with highest weight $\lv$, and \be{m} \mf{m} = \prod_{i=1}^{\ell} \prod_{j=1}^{\infty} \frac{ 1 - q^{- m_i -1} e^{-j \cc} }{1 - q^{ - m_i} e^{- j \cc} } \ee with $\cc$ denoting the minimal imaginary root of $G^{\vee}$, and $m_i, \, i=1, \ldots, \ell$ being the exponents of the underlying finite-dimensional root system. \end{nthm} 

Note both the similarity of this statement with the usual finite-dimensional Casselman-Shalika statement (see Theorem \ref{cs-fin}) and also the difference, namely the factor $\mf{m}$ interweaving the imaginary roots of the affine root system with the exponents of the underlying finite-dimensional one\footnote{The reader will also note a discrepancy between $\rho$ and $-\rho$ between these two formulas as well. This is accounted for by a choice of normalization of the measures involved as we will explain in \S \ref{sums:int}}. Its presence here is linked to a similar such appearance in the formulas for the affine Gindikin-Karpelevic and spherical functions.

\spoint In proving the above theorem, we were led to revisit the finite-dimensional Casselman-Shalika formula and find an argument for it circumventing certain standard facts  which are either not available or not yet available in the affine setting, e.g. longest elements in the Weyl group, uniqueness of Whittaker functionals.  In \S \ref{sec-fin}, we sketch a finite-dimensional proof along these lines. An important ingredient in this proof is a  recursion formula (\ref{whit:recur-body-1}) for Iwahori-Whittaker functions, one proof of which (see Appendix A) is based on integral formula of Harish-Chandra type (see Proposition \ref{hc-fla}) that links integrals over unipotent and compact groups.  In the affine setting, the analogous recursion result does hold though the proof give here in \S \ref{sec-rec} is more algebraic in form, utilizing certain intertwining and averaging operators. 

As with the original argument of Casselman-Shalika \cite{cs}, our proof uses crucially a passage to Iwahori-level objects. The work of Bernstein, Zelevinsky (and Lusztig) makes clear that certain Demazure-Lusztig type operators should govern the corresponding Iwahori-level combinatorics.  Actually, the connection between such operators and related $p$-adic groups may be traced a bit further back to Macdonald's fundamental computation of the spherical function for finite-dimensional groups \cite[Chap 4]{mac:mad}, where such operators are implicit.  Reinterpreting Macdonald's work using these operators was an important ingredient in our computation of the affine spherical functions (see \cite[\S7]{bkp}), as it allowed one to make the connection with some algebraic results of Cherednik \cite{cher:ma}. From this point of view, it was expected that an approach similar to the one taken in \cite[\S 6]{bkp} might be useful in computing Whittaker functions. However, the Demazure-Lusztig type operators used in the spherical case are not exactly the ones which are involved in the Whittaker picture.  We first learnt of the exact operators needed to compute Whittaker functions from \cite{bbl} (and a seminar talk by B. Brubaker), which provided an important impetus to the completion of this work. \footnote{The relation between the operators used here and those of \cite{bkp} is provided in \ref{T:T'}, as was observed already in \cite{bbl}.} The appearance of these operators in \emph{loc. cit} and our work is for the same reason-- to find a recursive description to Iwahori-Whittaker functions. However, whereas in \emph{loc. cit} a general uniqueness principle for Whittaker functionals is used to derive this recursion, here we proceed in a more direct manner as we have not yet established such a result in the affine setting. Note though that we do use a certain result (Lemma \ref{it-sum}(1)) which, in the finite-dimensional context, is a consequence of the uniqueness of Whittaker functions. It would be interesting to adapt the techniques of this paper to settings where uniqueness of Whittaker functions is known to fail.

Using the connection between certain Demazure-Lusztig operators and Iwahori-level Whittaker functions, the proof of the Casselman-Shalika formula is reduced to a combinatorial identity. Essentially the same ideas as in \cites{bkp, cher:ma} apply to establish this identity up to certain coefficient of proportionality (which turns out to be $\mf{m}$ from (\ref{m}) ). Let us note that, as in \cite{bkp}, to make the comparison between the algebraic and $p$-adic theories, one needs to invoke a rather subtle algebraic result of Cherednik (see Lemma \ref{sym:wd}) which ensures certain infinite symmetrizers have polynomial coefficients.  In the finite-dimensional case, the corresponding result is trivial. Moreover, in the finite-dimensional case another purely combinatorial argument (based on the existence of the longest element in the Weyl group) can be used to pin down the coefficient of proportionality. In the affine case,we approach the problem differently by making use of a limiting argument connecting the Whittaker and  Gindikin-Karpelevic integrals. Using our previous computation of the Gindikin-Karpelevic integral \cite{bgkp}, the unknown factor can be determined. A more refined combinatorial analysis should also be able to precise this factor and thereby avoid recourse to the affine Gindikin-Karpelevic formula (though one still needs to use the main finiteness result of \cite{bgkp} to even make sense of the Whittaker function).   

\tpoint{Acknowledgements} I would like to thank the organizers of the special semester at ICERM on "Automorphic Forms, Combinatorial Representation Theory, and Multiple Dirichlet Series" for providing a stimulating atmosphere in which part of this work was carried out, and especially G. Chinta, B. Brubaker, and A. Diaconu for several helpful and motivating discussions.  We also acknowledge IHÉS and the Max Planck Institute for Mathematics for providing excellent working conditions where parts of this work were written.

This paper owes a large intellectual debt to the works \cite{bkp, bgkp}, and the author owes an equally substantial debt to his collaborators  A. Braverman, H. Garland, and D. Kazhdan for numerous discussions and shared insights over the last few years, many of which have informed this work. In particular, many aspects of this work actually first arose in discussions with A. Braverman. 

\section{Finite Dimensional Case}  \label{sec-fin}

\spoint Let $\K$ be a non-archimedean local field with ring of integers $\O.$ Let $\pi \in \O$ be a uniformizer, and denote the residue field by $\kk:= \O / \pi \O,$ which is a finite field of size $q.$  Usually we shall denote algebraic groups (or functors of groups) over $\mc{K}$ (or a subring of $\mc{K}$) by boldface letters $\mathbf{X}, \mathbf{G}$ etc., and their sets of $\mc{K}$-points by $X,G$ etc.

\spoint \label{sec:fin-dim} Let $\mathbf{G}_o$ be a split, simple, and simply connected algebraic group (defined over $\zee$) and let $\mf{g}_o$ be its Lie algebra. Set $G_o=\mathbf{G}_o(\mc{K})$. Let $\mathbf{A}_o\subset \mathbf{G}_o$ be a maximal split torus, of rank $\ell$; we denote its character lattice by $\Lambda_o$ and its
cocharacter lattice by $\Lambda_o^{\vee}$; note that since we have assumed that
$\mathbf{G}_o$ is simply connected, $\Lambda_o^{\vee}$ is also the coroot lattice of $\mathbf{G}_o$.  For any $x \in \mc{K}^*$ and $\lv \in \Lv_o$
we set $x^{\lv}=\lv(x)\in A_o$. Let us choose a pair $\mathbf{B}_o,\mathbf{B}_o^-$ of opposite Borel subgroups such that $\mathbf{B}_o\cap\mathbf{B}_o^-=\mathbf{A}_o$. Let $\mathbf{U}_o$ and $\mathbf{U}_o^-$ denote the unipotent radicals of $\mathbf{B}_o$ and $\mathbf{B}_o^-$ respectively, so that we have semi-direct product decompositions $\mathbf{B}_o = \mathbf{A}_o \mathbf{U}_o$ and $\mathbf{B}_o= \mathbf{A}_o \mathbf{U}_o^-.$ Let  $R_o$ be the set of roots of $\mathbf{G}_o$ and by $R_o^{\vee}$ the set of coroots. Similarly $R_{o,+}$ (resp. $R^{\vee}_{o,+}$) will denote the set of positive roots (resp. of positive coroots), and $\Pi_o$ (resp. $\Pi_o^{\vee})$ the set of simple roots (resp. simple coroots). For each $a \in R_o$ we let $\mathbf{U}_a$ denote the corresponding one-parameter root subgroup (so $U_o$ is generated by $U_a$ for $a \in \Pi_o$). Let $\rho_o \in \Lambda_o$ denote the element such that $\la \rho_o, a^{\vee} \ra =1$ for all $a \in \Pi_o^{\vee}$ where $\la \cdot, \cdot \ra: \Lambda_o \times \Lv_o \rr \C$ denotes the natural pairing. Let $W_o$ be the Weyl group of $G,$ which is a finite Coxeter group with generators simple reflections $w_i:= w_{a_i}$  corresponding to $a_i \in \Pi_o.$ For a reductive group $\mathbf{H}$ over $\mc{K}$ we denote by $H^{\vee}$ the Langlands dual group defined over $\C$ defined by exchanging the root and coroot data of $\mathbf{H}.$

\spoint Let $K_o = \mathbf{G}_o(\O)$ be a maximal compact subgroup of $G_o.$ We set $A_{o, \O}:= A_o \cap K_o,$ and note that $A_o$ is a direct product $A_o(\O)A_o'$ where $\Lv \stackrel{\sim}{\rr} A'_o$ is the isomorphism which sends $\lv \in \Lv$ to $\pi^{\lv} \in A_o.$ We have the Iwasawa decomposition,  \be{iwa:+} G_o  =  K_o \, B_o  = K_o \, A_o \, U_o \ee with respect to which we may write  $g \in G_o$ as \be{iwa:+,e} g = k \pi^{\mv} u \text{ where } k \in K_o, \mv \in \Lv_o, u \in U_o. \ee  In such a decomposition, $\mv$ is uniquely defined, but the $K_o$ component $k$ is only well-defined modulo $K_o \cap B_o.$ Hence, we have well-defined maps \be{iwa:A+} \label{iwa:K+} \begin{array}{lcr}  \iw_{A_o}: G_o \rr A'_o, \; g \mapsto \pi^{\mv}  & \text{ and } & \iw_{K_o}: G_o \rr K_o/ K_o \cap B_o, \,   g \mapsto k \end{array}, \ee whenever $g$ is written as in (\ref{iwa:+,e}).

\spoint Denote by $\varpi_o: K_o \rr \mathbf{G}(\kk)$ the natural reduction map to the group over the residue field, and define the Iwahori subgroups of $K_o$ as follows, \be{iwahoris} \begin{array}{lcr} I_o:= \varpi^{-1}(\mathbf{B}(\kk)) &\text{ and } & I_o^- := \varpi^{-1}(\mathbf{B}^-(\kk)) \end{array}. \ee  Let us write $U_{o, \O} := U_o(\O)$ and \be{U:pi} U_{o, \pi}:= U_o(\pi \O) = \ker( \varpi_o: U_o(\O) \rr \mathbf{G}(\kk) ) \ee and similarly define $\um_{o, \pi}$ and $\um_{o, \O}.$  Consider now the following map $\gamma_o: \um_o \rr W_o$ defined by the commutative diagram, \be{hc-map} \xymatrix{ \um_o   \ar@/_/@{.>}[drr]_{\gamma_o} \ar[r] & \um_{o, \O} \setminus \um_o \ar[r]^{\iw_{K_o}} & \um_{o, \O} \setminus K_o / K_o \cap B_o \ar[d]^{\varpi_o}  \\
 & & \mathbf{U}^-_o(\kk) \setminus \mathbf{G}(\kk) / \mathbf{B}_o(\kk) = W_o
} \ee where the Bruhat decomposition is used to identity (as sets) $ \mathbf{U}^-_o(\kk) \setminus \mathbf{G}(\kk) / \mathbf{B}_o(\kk)$ with $W_o$ in the right-hand corner of the diagram. We have a decomposition $\um_o = \sqcup_{w \in W_o} U_o^{-, w},$ where we set           \be{Uo:w} U^{-, w}_o := \gamma_o^{-1}(w). \ee   This decomposition (or rather an affine version of it) already appeared in the work \cite[\S3]{bgkp} where it was used to prove certain certain finiteness results for $p$-adic loop groups.

\spoint \label{chars} For each $a \in \Pi_o$, let $\psi_{a}: U_{-a} \rr \C^*$ be a character of the one-dimensional additive root group. There is an isomorphism of abelian groups  \be{ab:um} \um_o/ [ \um_o, \um_o] \cong \prod_{a \in \Pi_o} U_{-a}, \ee  and so we may use the characters $\{ \psi_a \}_{a \in \Pi_o}$ to construct a map $\psi: \prod_{a \in \Pi_o} U_{-a} \rr \C.$ We denote the composition with the projection by the same letter, \be{phi:aff} \psi: \um_o \rr \um_o/ [ \um_o, \um_o] \rr \C^*.\ee We say that $\psi$ is \emph{principal} if all of the $\psi_a$ are non-trivial, and we say that $\psi$ is \emph{unramified} if all the $\psi_a$ are trivial on $\O$ but non-trivial on $\pi^{-1} \O \subset \mc{K}.$  

\spoint Let $\C[\Lv_o]$ denote the group algebra of the cocharacter lattice of $G_o$, with generators $e^{\lv}, \lv \in \Lv_o.$ Also, let $d\nm$ denote the Haar measure on $U^-_o$ which which gives the compact subgroup $K_o \cap U^-_o$ volume $1.$ Then we define the \emph{unramified Whittaker function} $\W_o: G \rr \C[\Lv_o]$ via \be{whit:fin} \W_o(g) = \int_{U^-_o} \Phi_{\rho_o}(g \nm) \psi(\nm) d\nm \ee where if $\nu \in \Lambda_o$ and $\iw_{A_o}(g) = \pi^{\mv}, \; \mv \in \Lv_o$ we set \be{Phi:defn} \Phi_{\nu}(g) = e^{\mv}q^{\la \nu, \mv \ra} \in \C[\Lv_o]. \ee We note immediately that $\W_o$ is left $K_o$-invariant and satisfies the following property, henceforth referred to as \emph{ right $(\um, \psi)$-invariance} \be{W:left} \W_o(g \nm) = \psi(\nm) \W_o(g) \text{ if } \nm \in U^-_o. \ee  Thus $\W_o$ is determined by its values on the $\Lv \cong A_o';$ moreover, a standard argument shows that left $K_o$-invariance combined with right $(\um, \psi)$-invariance implies $\W_o(\pi^{\lv}) =0$ unless $\lv \in \Lv_+.$

\tpoint{Whittaker Integrals and Whittaker Sums} \label{sums:int} We now try to explain the discrepancy between "$\rho$" and "$-\rho$" that occurs within the affine and finite Casselman-Shalika formula (see Theorems \ref{cs-fin} and \ref{main}). Let $\lv \in \Lv_{o, +}$ and $\mv \in \Lv_o$ and consider the set \be{M-two} M_o(\lv; \mv)= K_o \setminus K_o \pi^{\lv} \um_o \cap K \pi^{\mv} \up_o. \ee Since $\lv \in \Lv_+$ we have a well-defined map \be{niw-fin} \niw: K_o \pi^{\lv} \um_o \rr \um_{o, \O} \setminus \um_o \ee (see comments after (\ref{niw:def-a}) below), and hence a map $\niw: M_o(\lv; \mv)  \rr \um_{o, \O} \setminus \um_o.$
Now, starting from (\ref{whit:fin}) for $g= \pi^{\lv},\, \lv \in \Lv_+$ a simple change of variables (namely, $n^{-} \mapsto \pi^{-\lv} \nm \pi^{\lv}$) shows  \be{W:lv} \W_o(\pi^{\lv}) &=& \int_{\um_o} \Phi_{\rho_o}(\pi^{\lv} \nm) \psi(\nm) d\nm = \label{whit:tr} \Phi_{- \rho_o}(e^{\lv}) \int_{\um_o} \Phi_{\rho_o}(\nm) \psi( \pi^{-\lv} \nm \pi^{\lv} ) d\nm. \ee  By the definition of the Haar measure $d\nm,$ the last expression can be written as a sum, 
\be{whit:fin-sum} \W_o(\pi^{\lv}) &=& \Phi_{- \rho_o}(e^{\lv})  \sum_{\mv \in \Lv} e^{\mv} q^{ \la \rho, \mv \ra} \sum_{x \in M_o(0; \mv) } \psi(\pi^{-\lv} \niw(x) \pi^{\lv} ). \ee There is a bijection $r_{\lv}: M_o(0;\mv) \rr M_o(\lv; \mv + \lv)$ such that \be{r:bij} \niw(r_{\lv}(x)) = \pi^{\lv} \niw(x) \pi^{-\lv}, \text{ for } x \in M_o(0;\mv).\ee Hence, the last sum may also be written as  
\be{whit:fin-sum2} (\ref{whit:fin-sum}) &=& \Phi_{- \rho_o}(e^{\lv})  \sum_{\mv \in \Lv} e^{\mv} q^{ \la \rho, \mv \ra} \sum_{x \in M_o(\lv; \mv + \lv) } \psi(\niw( x) ) \\ &=& q^{ -\la 2 \rho_o, \lv \ra}
\sum_{\mv \in \Lv} e^{\mv} q^{ \la \rho, \mv \ra} \sum_{x \in M_o(\lv; \mv) } \psi(\niw( x) )\ee  Omitting the factor of $q^{ -\la 2 \rho_o, \lv \ra}$ from the latter sum, we may define \be{wfs:3} \wt{\W}_o(\pi^{\lv}) := \sum_{\mv \in \Lv} e^{\mv} q^{ \la \rho, \mv \ra} \sum_{x \in M_o(\lv; \mv) } \psi(\niw( x) ). \ee This has an integral expression similar to (\ref{whit:tr}) as well, but one in which we normalize the measure $d\nm$ so that $\pi^{-\lv} \um_{\O} \pi^{\lv}$ is assigned measure $1.$ 

In the affine setting, there is no natural Haar measure on the unipotent radical such that the change of variables formula (\ref{whit:tr}) holds, i.e. and we cannot compare the measure which assign volume $1$ to $\pi^{-\lv} \um_{\O} \pi^{\lv}$ and the one which assigns volume $1$ to $\um_{\O}$ (\emph{heuristically}, they should be related by a factor of $q^{ - \la 2 \rho, \lv \ra}$). On the other hand, using the finiteness results \cite[Theorem 1.9(1)]{bgkp}, sums such as (\ref{whit:fin-sum2}) and (\ref{wfs:3}) do make sense. We find (\ref{wfs:3}) more natural in the affine context, and take it as the basis of our definition of the affine Whittaker function (see Definition \ref{whitfn} below.)

\newcommand{\G}{\mc{G}}

\spoint Omitting the character $\psi$ and dividing by $\Phi_{- \rho_o}(e^{\lv})$ in (\ref{whit:tr}), one obtains the \emph{Gindikin-Karpelevic} integral, \be{gk:fin}   \mc{G}_o:=  \int_{U^-_o} \Phi_{\rho_o}(\nm) d\nm = \sum_{\mv \in \Lv} e^{\mv} q^{\la \rho_o, \mv \ra} \, | M_o(0; \mv) |  \ee where $| S |$ denotes the cardinality of a (finite) set $S.$ This sum can be explicitly computed (see \cite{bfk}) \be{gk-value} \mc{G}_o = \prod_{a \in R_{o, +}} \frac{1 - q^{-1}e^{-a^{\vee}}}{1 - e^{-a^{\vee}}}. \ee If $f \in \C[\Lv_o]$ is written as $f = \sum_{\mv \in \Lv_o} c_{\mv} e^{\mv}$ we shall denote its $e^{\mv}$ coefficient as \be{coeff:f} [e^{\mv}] f = c_{\mv}. \ee  We then have the following result, whose proof (in the affine context) is given in Proposition \ref{gk-whit}..

\begin{nlem} \label{approx-intro-lem} For any fixed $\mv \in \Lv_o$ there exists $\lv \in \Lv_{o,+}$ sufficiently dominant so that \be{approx-intro} [e^{\mv}] \mc{G}_o = q^{ \la \rho_o, \lv \ra} [e^{\mv - \lv}] \W_o(\pi^{\lv}) \ee \end{nlem} \noindent

\spoint The main fact about $\W_o(g)$ is contained in the following result, generally known as the Casselman-Shalika formula \footnote{Langlands conjectured a version of this formula in a letter to Godement \cite{lang-godement}. For $GL_n$, such a formula was first proven by Shintani \cite{shin}, and then for other classical types by Kato \cite{kato}. At roughly the same time Casselman and Shalika \cite{cs} established the general result using techniques quite distinct from those of Shintani and Kato. We thank Kato for sending us his unpublished manuscript.}

\begin{nthm} [\cite{cs}] \label{cs-fin} Suppose that $\lv \in \Lv_{o, +},$ and denote by $\chi_{\lv} \in \C[\Lambda_o^{\vee}]$ the corresponding Weyl character of the highest weight representation of the dual group $G^{\vee}$. Then \be{cs-fin:s}  \W_o(\pi^{\lv}) = q^{ - \la \rho_o, \lv \ra} \prod_{a \in R_{o, +}} ( 1 - q^{-1} e^{-a^{\vee}}) \;  \chi_{\lv}  \ee 
\end{nthm} 

Let us now sketch a proof of this theorem which generalizes to the setting of loop groups. 

\tpoint{Step 1: Decomposition and Recursion} We first note that the decomposition $\um_o = \sqcup_{w \in W} \umw_o$ (see the line before (\ref{Uo:w})) results in the following expression: \be{W:wsum} \W_o(\pi^{\lv}) = \sum_{w \in W} \W_{o,w}(\pi^{\lv}) \ee where $\W_{o, w}(\pi^{\lv})$ are the \emph{Iwahori-Whittaker functions}, defined as \be{W:wpart} \W_{o,w}(\pi^{\lv}) :=\int_{U^{-,w}_o} \Phi_{\rho_o}(\pi^{\lv} \nm) \psi(\nm) d\nm = \Phi_{- \rho_o}(e^{\lv}) \int_{U^{-,w}_o} \Phi_{\rho_o}(\nm) \psi(\pi^{-\lv} \nm \pi^{\lv} ) d\nm. \ee

\begin{nprop} \label{whit:recur-body-1} Let $w = w' w_a$ with $a \in \Pi_o$ a simple root, and $\ell(w) = 1 + \ell(w').$ Then we have the following identity in $\C[\Lv_o],$ \be{W:rec-1} \W_{o, w} (\pi^{\lv}) = \frac{1 - q^{-1} e^{-\av}}{1 - e^{\av}} \W_{o, w'}(\pi^{\lv})^{w_a} + \frac{q^{-1} -1 }{1 - e^{\av}} \W_{o, w'}(\pi^{\lv}), \ee where $\W_{o, w'}(\pi^{\lv})^{w_a}$ is the termwise application of $w_a$ to the expression $\W_{o, w'}(\pi^{\lv}) \in \C[\Lv]$ and the rational functions in the above expression are expanded in positive powers of the coroots (a cancellation occurs to ensure the whole expression is in $\C[\Lv]$). \end{nprop}

This proof is carried out in Appendix \ref{app-hc}. Alternatively, in the affine setting we give a different proof (which also works in the finite-dimensional setting) using intertwining and averaging operators.

\tpoint{Step 2: Rephrasal using Demazure-Lusztig Operators}  The above Proposition \ref{whit:recur-body-1}
can be reformulated with the aid of certain elements in $\C(\Lv_o)[W_o]$, the group algebra of $W_o$ over the ring of rational functions on $\C(\Lv_o).$ Namely, for each $a \in \Pi_o$ we define \be{T:a} T_a := \c(a) [w_a] + \b(a) [1] \,  \ee where \be{b:c} \begin{array}{lcr} \c(a) = \frac{1 - q^{-1} e^{-\av}}{1 - e^{\av}} & \text{ and } \b(a) := \frac{q^{-1} -1 }{1 - e^{\av}}  \end{array}. \ee The elements $T_a, a \in \Pi_o$ satisfy the relations for a Hecke algebra $\mathbb{H}_o$ associated to the finite group $W_o$ (see Proposition \ref{quad:braid}). In particular, since the braid relations are satisfied, if $w \in W_o$ has a reduced decomposition $w = w_{a_1} \cdots w_{a_r}, a_i \in \Pi_o$ we may define $T_w:= T_{a_1} \cdots T_{a_r}$ and note that the above definition does not depend on the reduced decomposition chosen. One can check that the natural action of $W_o$ on $\Lv_o$ extends to give an action of $\mathbb{H}_o$ on $\C[\Lv_o]$ via the operators $T_a$ defined above. Its relevance for us here is that the Iwahori-Whittaker functions may be expressed as \be{whit:recur:polyrep} \W_{o,w} (\pi^{\lv} ) = q^{ -\la \rho_o, \lv \ra} T_{w^{-1}} (e^{\lv}). \ee In case $w=1$ we can easily verify that $\W_{o, 1}(e^{\lv})= q^{ - \la \rho_o, \lv \ra} e^{\lv}.$ Thus, inductively all $\W_{o, w}(\pi^{\lv})$ may be computed recursively.

\tpoint{Step 4: Reassembly and the Gindikin-Karpelevic Limit} From (\ref{W:wsum}) and (\ref{whit:recur:polyrep}), we see the computation of the Whittaker function reduces to a combinatorial problem, namely the identification of the right hand side of the following identity, \be{whit:polyrep}   \W_o(\pi^{\lv}) =  q^{ -\la \rho_o, \lv \ra} \sum_{w \in W_o} T_{w^{-1}} (e^{\lv})= q^{ -\la \rho_o, \lv \ra} \sum_{w \in W_o} T_{w} (e^{\lv}). \ee In other words,  Theorem \ref{cs-fin} amounts to the identity \be{cs:iden}    \sum_{w \in W_o} T_w (e^{\lv}) = \prod_{a \in R_{o, +}} ( 1 - q^{-1} e^{-a^{\vee}}) \;  \chi_{\lv}. \ee This follows in turn from the following "operator Casselman-Shalika" formula, i.e., the following algebraic identity in $\C(\Lv_o)[W_o]$ \be{alg:cs} \sum_{w \in W} T_w = \prod_{a \in R_{o, +}} ( 1 - q^{-1} e^{-a^{\vee}}) \sum_{w \in W} [w] \prod_{a \in R_{o, +}}  \frac{1}{ ( 1 -  e^{-a^{\vee}})}. \ee The proof of this identity rests on the Hecke properties of $T_a,$ and is explained in more detail in Proposition \ref{alg-iden}. The main idea is to show that both sides are eigenfunctions of each $T_a,$ $a \in \Pi_o$ with the same eigenvalue. This ensures that the two sides of the above equation are proportional, i.e., we can write \be{whit-ctp} \W_o(\pi^{\lv}) = q^{- \la \rho_o, \lv \ra } \sum_{w \in W} T_w(e^{\lv}) =  q^{- \la \rho_o, \lv \ra}  \ \mf{c}  \prod_{a \in R_{o, +}} ( 1 - q^{-1} e^{-a^{\vee}}) \sum_{w \in W} [w] \prod_{a \in R_{o, +}}  \frac{1}{ ( 1 -  e^{-a^{\vee}})} \ee where $\mf{c}$ is some element in $\C(\Lv_o)$ which is independent of $\lv.$ Now Lemma \ref{approx-intro-lem} tells us that the Whittaker function resembles the Gindikin-Karpelevic integral for large values of $\lv.$ Since the latter integral can be computed explicitly (see (\ref{gk-value})), we can show that $\mf{c}=1$ (for details on this point see \S \ref{det-ct}).  

\section{Notations on Affine Lie Algebras and Groups}  \label{sec-affnot}

We will review in this section some basic facts about affine Kac-Moody algebras and the corresponding $p$-adic groups associated to them. The notation is the same as in \cite{bgkp, bkp} and we refer the reader to these papers (and the references therein) for further details on the constructions we sketch here.

\subsection*{Affine Lie algebras and root systems}
\noindent

\tpoint{Affine Lie Algebras} We maintain the same finite-dimensional conventions as in \S \ref{sec:fin-dim}. For a field $k,$ we  denote  by $\f{g}$ the (untwisted) affinization of the Lie algebra $\f{g}_o.$ As a vector space $\f{g}:= k \dd \oplus \mf{g}'$ where $\dd$ is the degree derivation and $\f{g}'$ is the one-dimensional central extension of the the loop algebra $ \f{g}_o \otimes_k k[t, t^{-1}]$ which is specified by normalizing the Killing form  $(\cdot, \cdot)$ on $\mf{g}_o$ so that the highest root has length $2.$ Let $\mf{h} \subset \mf{g}$ denote a Cartan subalgebra containing the finite-dimensional Cartan $\mf{h}_o,$ the degree derivation $\dd,$ and the center $\mf{h}_{cen}$ of $\mf{g},$  which is one-dimensional and spanned by $\cc$, the minimal imaginary coroot, i.e., one has a direct sum decomposition $\mf{h}:= \mf{h}_o \oplus k \cc  \oplus k \dd .$   Let $\f{h}^*$  be the algebraic dual of $\f{h}$. As before we denote by  $ \la \cdot, \cdot \ra \,: \f{h}^* \times \f{h} \rightarrow k$ the natural pairing. Let $R$ be the set of roots of $\f{g},$ and $R^{\vee}$ the set of coroots. We denote the set of simple roots of $\mf{g}$  by $\Pi= \{ a_1, \ldots, a_{\ell+1} \} \subset \f{h}^*$ and write $\Pi^{\vee}  =  \{ a_1^{\vee}, \ldots, a_{\ell+1}^{\vee} \}   \subset  \f{h}$ for the set of simple affine coroots.

 For each $i=1,\ldots, \ell+1$ we denote by $w_{a_i}$ (or just sometimes $w_i$) the corresponding simple reflection and denote by  $W \subset \Aut(\f{h})$  the group generated by the elements $w_i$ for $i=1, \ldots, \ell.$ It is a Coxeter group.
 
 A root $a \in R$ is called a \emph{real root} if there exists $w \in W$ such that $w a \in \Pi.$ The set of such roots is denoted as $R_{re}.$ Otherwise, $a$ is called an \emph{imaginary root}, and the set of all such imaginary roots is denoted $R_{im}.$ For each $a \in R$ we let $m(a)$ denote the multiplicity of $a:$ it is equal to $1$ is $a \in R_{re}$ and $\ell$ if $a \in R_{im}.$

We define the affine root lattice $Q$ as the $\zee$-module spanned by $\Pi$ and the affine coroot lattice $Q^{\vee}$ as the $\zee$-module spanned by $\Pi^{\vee}.$ Denote the subset of non-negative integral linear combinations of the affine simple roots  (respectively, affine simple coroots) as $Q_+$ (respectively,  $Q^{\vee}_+$).  The integral weight lattice is defined  by
\begin{equation}\label{integralweightlatticedef}
\Lambda  \ \  :=  \ \  \{ \lambda \in \f{h}^* \ | \   \langle \lambda, a_i^{\vee} \rangle \,  \in \,  \zee  \, \ \text{ for } \, i\,=\, 1, \ldots, \ell+1  \text{ and } \la \lambda, \dd \ra \in \zee \}\,.
\end{equation}  and its dual, the coweight lattice $\Lambda^{\vee}$ is defined as \be{aff-cowts} \Lv:= \{ \lv \in \f{h}^* \mid \la a_i, \lv \ra \in \zee \text{ for } i=1, \ldots, \ell+1 \text{ and } \la \Lambda_{\ell+1}, \lv \ra \in \zee  \} .\ee The set of dominant weights $\Lambda_+$ is defined as \be{lv_+} \Lambda_+  :=  \ \  \{ \lambda \in \f{h}^* \ | \   \langle \lambda, a_i^{\vee} \rangle \, \geq  \,  0  \, \ \text{ for } \, i\,=\, 1, \ldots, \ell+1  \text{ and } \la \lambda, \dd \ra \in \zee \}\,, \ee and the set of dominant coweights $\Lv_+$ is defined analogously.  Finally we define $\rho \in \Lambda$ by the following conditions \be{rho:v} \la \rho, a_i^{\vee} \ra = 1 \text{ for } i =1, \ldots, \ell+1 \text{ and } \la \rho, \dd \ra = 0. \ee We define $\rho^{\vee} \in \Lv$ by replacing in the above definition the simple coroots with the simple roots and $\dd$ with $\Lambda_{\ell+1}.$

Let $\mf{g}^{\vee}$ denote the dual Lie algebra to $\mf{g}$ (obtained by the transpose of the generalized Cartan matrix of $\mf{g}$). In general $\mf{g}^{\vee}$ is again an affine Lie algebra, but could be of twisted type. If $\mf{g}_o$ is of simply-laced type, then $\mf{g}^{\vee}$ is the untwisted affine Lie algebra attached to $\mf{g}_o^{\vee},$ the dual of the underlying finite-dimensional root system. To avoid the complications of twisted affine types, \emph{we shall throughout restrict to the case that $\mf{g}_o$ is of simply-laced type.}

\tpoint{Modules for Affine Lie Algebras} \label{loo}  Given $\lambda, \mu \in \f{h}^*$ we define the dominance partial order on $P(V)$ by \be{dom:wts} \lambda \geq \mu \iff \lambda - \mu \in Q_+. \ee For any $\lambda \in \Lambda_+$ we let $V^{\lambda}$ denote the irreducible highest weight module of $\mf{g}$ with highest weight $\lambda;$ it is equipped with a decomposition \be{V:dec} V^{\lambda} = \oplus_{\mu \in P(V^{\lambda}) } V^{\lambda}(\mu) \ee where each $V^{\lambda}(\mu)$ is finite-dimensional and in fact $\mu \leq \lambda.$ We next formulate the Weyl-Kac character formula for $V^{\lambda}$, but first we need some preliminaries on defining the ring in which the character takes its value.

\tpoint{Dual Coweight Algebras} Let $S$ be any ring, and consider the set of formal linear combinations \be{f:lc} f:= \sum_{\lv \in \Lv} c_{\lv} e^{\lv}, \text{ where } c_{\lv} \in S. \ee The support of $f$ as above will be the set $\Supp(f) \subset \Lv$ consisting of all $\mv \in \Lv$ such that  $c_{\mv} \neq 0.$ We shall say that $f$ has semi-infinite support if there exists $\lv_1, \ldots, \lv_r \in \Lv_+$ such that \be{cat:O} \Supp(f) \subset \mf{o}(\lv_1) \cup \cdots \cup \mf{o}(\lv_r) \ee where \be{o} \mf{o}(\lv) = \{ \mv \in \Lv | \mv \leq \lv \},\ee and $\leq$ denotes the dominance order on coweights (defined as in (\ref{dom:wts}) ). Let us define \be{loo:R} S_{\leq}[\Lv]:= \{ f=  \sum_{\lv \in \Lv} c_{\lv} e^{\lv} \mid \Supp(f) \text{ is semi-infinite} \}. \ee By \cite{loo}, the set $S_{\leq}[\Lv]$ is a ring, which carries a natural action of $W$ (via $w e^{\lv} = e^{w \lv}$ for $w \in W$ and $\lv \in \Lv$). For $f \in S_{\leq}[\Lv]$ we denote by $f^w$ the application of $w \in W$ to $f$.

\tpoint{Weyl-Kac Characters} Let us now define  \be{del} \mf{D} := \prod_{ a \in R_+} (1 - e^{-\av})^{m(\av)}  \ee where $m(\av)$ is the multiplicity of the root $\av.$ One can easily show that $\mf{D} \in \C_{\leq}[\Lv]$ and that moreover (see \cite[p. 172, 10.2.2]{kac}) one has the following identity in $\C_{\leq}[\Lv],$\be{del:w} e^{w \rho} \mf{D}^w = (-1)^{\ell(w)} e^{\rho} \mf{D}. \ee The Weyl-Kac character formula now states, 

\begin{nthm}\cite[Theorem 10.4, p. 173]{kac} \label{weyl-kac} Let $\lv \in \Lv_+$ and let $\chi_{\lv}:= \sum_{\mv \in \Lv} \dim_{\C} V^{\lv}(\mv) e^{\mv} $ be the character of the irreducible highest weight representation (of $\mf{g}^{\vee}$) with highest weight $\lv.$ Then $\chi_{\lv} \in \C_{\leq}[\Lv]$ and we have the following equality in $\C_{\leq}[\Lv],$ \be{wk-fla} \chi_{\lv} = \frac{1}{\mf{D}} \sum_{w \in W} (-1)^{\ell(w)} e^{w (\lv + \rho) - \rho} = \sum_{w \in W} \frac{ e^{w \lv} }{\mf{D}^w } . \ee \end{nthm} Note that implicit in the above theorem is the assertion that all expressions in (\ref{wk-fla}) lie in $\C_{\leq}[\Lv].$ The equality of the second and third expression follows at once from (\ref{del:w}).

\subsection*{Affine Kac-Moody Groups} \label{section-loopgroups} \noindent

\tpoint{The Tits Group Functor}  Let $\mathbf{G}$ denote the affine Kac-Moody group functor of Tits constructed in \cite{tits} (or see \cite{bkp} where this construction is reviewed in the above notation). Recall that the construction involves, for any  $a \in R_{re}$ a corresponding one-dimensional additive group scheme $\mathbf{U}_a$ together with a fixed isomorphism $x_a: \mathbb{G}_a \rr \mathbf{U}_a.$ We let $\mathbf{U}$ and $\mathbf{U}^-$ denote the subgroup scheme generated by $\mathbf{U}_a$ for $a \in R_{re, +}$ or $R_{re ,-}$ respectively. For each  $i=1, \ldots, \ell+1$ we choose isomorphisms $x_{a_i}: \mathbb{G}_a \rr \mathbf{U}_{a_i}$ and $x_{-a_i}: \mathbb{G}_a \rr \mathbf{U}_{-a_i}$ and for each invertible element $r \in S^*$ (here again $S$ is some test ring) and $i=1, \ldots, \ell+1$ denote by $\wt{w_i}(r)$ the image of the product \be{w_i(r)}  x_{a_i}(r) x_{-a_i}(-r^{-1}) x_{a_i}(r)  \ee in $\mathbf{G}(S).$ We set $\wt{w_i}:= \wt{w_i}(1).$ Also, recall that there exists an embedding $\mathbf{A} \subset \mathbf{G}$ where $\mathbf{A}$ is the functor that sends $S$ to $\mathbf{A}(S)= \Hom_{\zee}(\Lambda, S).$ For $u \in S^*$ and $\lv \in \Lambda^{\vee} $ we write $s^{\lv}$ for the element of $\mathbf{A}(S)$ which sends each $\mu \in \Lambda$  to $s^{\langle \mu, \lv \rangle} \in S.$

Now we describe the structure of $G:= \mathbf{G}(k)$ for any field $k$. For each $a \in R_{re}$ we define $U_a = \mathbf{U}_a(k),$ and we also let $A= \mathbf{A}(k).$  \footnote{Recall our convention: if $\mathbf{X}$ is some functor, we shall denote by the roman letter $X$ the set of points $\mathbf{X}(k)$ over some field $k.$ }
Let $U$ denote the subgroup generated by $U_a$ for $a \in R_{re, +}$ and $U^-$ the subgroup generated by $U_a$ for $a \in R_{re, -}.$ Define now $B_a$ to be the subgroup of $G$ generated by $U_a$ and $T.$ Also, set $B$ and $B^-$ to be the subgroups generated by all the $B_a$ for $a \in R_{re, +}$ and $R_{re, -}$ respectively.  We have semi-direct products $B = A \rtimes U$ and $B^- = A \rtimes U^-.$ We let $N$ be the group generated by $T$ and the $\wt{w_i}$ defined as above. There is a natural map $\zeta: N \rr W$ which sends $\wt{w_i} \mapsto w_i$ and  which has kernel $A.$ This map is surjective, and induces an isomorphism $\zeta: N/A \rr W.$  For each $w \in W,$ we shall write $\dot{w}$ for any lift of $w$ by $\zeta.$ If $w \in W$ has a reduced decomposition $w = w_{a_{i_1}} \cdots w_{a_{i_r}},$ with the $a_k \in \Pi,$ we shall also sometimes write \be{tw} \wt{w}:= \wt{w}_{i_1} \cdots \wt{w}_{i_r} \ee for a specific lift of $w$ (depending on the reduced decomposition) where the $\wt{w}_i$ were defined after (\ref{w_i(r)}). One has the following Bruhat-type decompositions \be{bruhat} G &=& \sqcup_{w \in W}  B \, \dot{w} \, B \ \ = \ \  \sqcup_{w \in W} \, B^- \dot{w} \,  B^- \\
&=& \sqcup_{w \in W} B^-\,  \dot{w} \, B \ \ = \ \ \sqcup_{w \in W} B \,  \dot{w} \, B^-,\ee where $\dot{w}$ is any lift of $w \in W$ to $N$ under the map $\zeta$ above. Note that it is important here that we are working with the so-called minimal Kac-Moody group in order to have Bruhat decompositions with respect to $B$ and $B^-.$ We shall denote the \emph{big cell} of $G$ as $BU^-.$ Note that if $g \in BU^-$ then we may uniquely write $g = u h u^-$ with $u \in U, h \in A, u^- \in U^-$. 

\subsection*{Recollections on $p$-adic Loop Groups} \noindent

\spoint \label{padic-basic} Now we collect a few useful definitions and constructions related to $G:= G(\mc{K}),$ a group which we informally refer to as a \emph{$p$-adic loop group} in place of the more precise terminology \emph{affine Kac-Moody group over a non-archimedean local field}. Throughout we set \be{K} K:= \mathbf{G}(\O) \subset G \ee which plays the role of a maximal compact subgroup in the discussion below (note that we have not defined any topology on $G,$ so this is only an analogy). We set $A := \mathbf{A}(\mc{K}) \cong \Hom_{\zee}(\Lambda, \mc{K}) \subset G.$ Let $A_{\O}:= \mathbf{A}(\O)$ and note that we have an direct product decomposition $A = A' A_{\O}$ where $A'$ is identified with $\Lv$ via the map $\lv \mapsto \pi^{\lv}$ with $\pi \in \O$ the fixed uniformizing element. For each $a \in R_{re},$ recall that the elements of the one-dimensional group $U_a:= \mathbf{U}_a(\K)  \subset G$ are written as $x_a(u)$ for $u \in \K.$ For $m \in \zee$ we set \be{U:m} U_{(a, m)}:= \{ x_a(u) | v (u) \geq m \} .\ee  As a shorthand, if $a \in R_{re}$ we write \be{Ua:defs} \begin{array}{lcr} U_{a, \O} := U_{(a, 0)} = \mathbf{U}_a(\O) &
U_{a, \pi} := U_{(a, 1)} &
U_{a}[m] := U_{(a, m)} \setminus U_{(a, m-1)} \end{array} \ee Let us also set $U_{\pi}$ to be the group generated by $U_{a, \pi}$ with $a \in R_{re, +}$ and $U^-_{\O}$ the group generated by $U_{-a, \O}$ for $a \in R_{re, +}.$ Similarly, we may define the groups $U_\O$ and $U^-_\pi.$ The group $K = \mathbf{G}(\O)$ is generated by the subgroups $U_{a, \O}.$  Letting $\varpi: K \rr G_\res$  denote the map induced from the natural reduction $\O \rr \res$ (recall that $\res$ denotes the residue field). We shall define the \emph{Iwahori subgroups} $I, I^- \subset K$ as \be{iwa:+} \begin{array}{lcr} I:= \{  x \in K | \varpi(x) \in B_\res \} & \text{ and } & I^-:= \{ x \in K \mid \varpi(x) \in B^-_{\res} \}. \end{array} \ee Choose representatives $\wt{w} \in K$ for $w \in W$ as in (\ref{tw}), we have a decomposition  \be{K:im}  K &=& \sqcup_{w \in W}  I^- \, \wt{w} \, I. \ee  The Iwasawa decomposition states that \be{iwa:K} G = \sqcup_{ \lv \in \Lv} K \pi^{\lv} \up = \sqcup_{\lv \in \Lv} K \pi^{\lv} \um. \ee In general, for $g \in G$ the $K$ and $\up$ (or $\um$) component are not uniquely determined. Let $\aw:= \Lv \rtimes W$ be the "affinized" Weyl group: for each $x = (\lv, w) \in \aw,$ we denote by $\pi^{\lv} \wt{w} \in G$ again by the same letter $x.$ The following is a consequence of the Iwasawa decomposition and (\ref{K:im}): \be{G:im} G = \sqcup_{x \in \aw} A_{\O} U x I^- \ee and we refer to \cite[\S3]{bkp} for a proof, which follows from the Iwasawa decomposition for $G$ and the Iwahori-Matsumoto decomposition for $K.$

\section{Whittaker Sums } \label{sec-whit}

Throughout this section, we maintain the conventions of \S \ref{padic-basic}.

\subsection*{Basic Definitions} \noindent

\spoint \label{chars} For each $a \in \Pi$, let $\psi_a: U_a \rr \C^*$ be a character. There is an isomorphism of abelian groups  \be{ab:um} \um/ [ \um, \um] \cong \prod_{a \in \Pi} U_a(\mc{K}), \ee  and so the characters $\{ \psi_a \}_{a \in \Pi}$ yield a map $\psi: \prod_{a \in \Pi} U_a(\mc{K}) \rr \C^*.$ We denote the composition with the natural projection by the same letter, \be{phi:aff} \psi: \um \rr \um/ [ \um, \um] \rr \C^*.\ee We say that $\psi$ is \emph{principal} if all of the $\psi_a$ are non-trivial, and we say that $\psi$ is \emph{unramified} if all the $\psi_a$ are trivial on $\O$ but non-trivial on $\pi^{-1} \O \subset \mc{K}.$  

\spoint To motivate the definition of the affine Whittaker function, we would like to study functions which are left $K$-invariant which satisfy the analogue of the condition (\ref{W:left}). As in the finite-dimensional case, this already imposes the following dominance condition, whose proof we omit.

\begin{nlem} \label{dom-lem} Fix $\psi$ a principal, unramified character of $U^-$ as in \S \ref{chars}. Let $f:G \rr \C_{\leq}[\Lv]$ be any function which if left $K$-invariant and satisfies the following condition (known hereafter as $(\um, \psi)$-invariance) \be{K:f:n} f (g n^-) = \psi(n^-) f(g) \text{ for } n^- \in N^- \text{ and } g \in G. \ee Then $f$ is determined by its values on $\Lv \cong A' \subset G$ and moreover $f(\pi^{\lv})=0$ unless $\lv \in \Lv_+.$ \end{nlem}

Before we present the full definition of the affine Whittaker function, we make the following simple observation: if $\lv \in \Lv_+$, then there is a map from the set \be{niw:def-a} \niw: K \pi^{\lv} \um \rr \um_{\O} \setminus \um \ee  specified as follows. Given any $x \in K \pi^{\lv} \um$ which we write as  $x=k \pi^{\lv} \nm \in K \pi^{\lv} \um$ for $k \in K, \nm \in U^-,$ recall that  $\nm$ is only uniquely specified up to left multiplication by elements of the form $\pi^{-\lv} \um_{\O} \pi^{\lv}.$ Since $\lv \in \Lv_+,$ we have $\pi^{-\lv} \um_{\O} \pi^{\lv} \subset \um_{\O}$  and so the map $\niw(x) = \um_{\O} \nm$ is well-defined. Using this remark and the above Lemma, we now present the following

\begin{de} \label{whitfn} Let $\psi$ be an unramified, principal character of $\um.$ The Whittaker function $\W: G \rr \C_{\leq}[\Lv]$ is defined to be the unique function satisfying the following two conditions, 

\begin{enumerate}

\item  $\W(k g \nm) = \psi( \nm ) \W(g)$ for $g \in G, \, \nm \in \um, \, k \in K.$
\item For $\lv \in \Lv_+$, $\W(\pi^{\lv})$ is given by the sum \be{whit:diag} \W(\pi^{\lv}) = \sum_{\mv \in \Lv} e^{\mv} q^{ \la \rho, \mv \ra} \sum_{x \in M(\lv; \mv) } \psi(\niw(x) ) \ee where $M(\lv; \mv) = K  \setminus K \pi^{\lv} \um \cap K \pi^{\mv} U,$ and $\niw: M(\lv; \mv) \rr U^-_{\O} \setminus U^-$ is the map defined as in (\ref{niw:def-a}).  

\end{enumerate}

\end{de} 

\rpoint{Remarks} Let us note the following:

\begin{enumerate} 

\item[i.] The previous Lemma \ref{dom-lem} implies that such a function $\W$ is uniquely specified by conditions (1) and (2), which are easily seen to be compatible. 

\item[ii.] In condition (2), since $\psi$ is trivial on $\um_{\O}$ it descends to a well-defined function on $\niw(M(\lv; \mv)).$ 

\item[iii.] From \cite[Theorem 1.9]{bgkp} we have that $M(\lv; \mv)$ is always equal to a finite set. It is equal to the empty set unless $\mv \leq \lv.$  Hence, $W(\pi^{\lv})$ is well-defined and takes values in $\C_{\leq}[\Lv].$

\item[iv.] We may also write $\W(\pi^{\lv})$ as a sum over $M(0; \mv) = K  \setminus K \um \cap K \pi^{\mv} U,$ as follows,
\be{W-alt}  \W(\pi^{\lv}) =  q^{\la \rho, \lv \ra} e^{\lv} \sum_{\mv \in \Lv} e^{\mv} q^{ \la \rho, \mv \ra} \sum_{x \in M(0; \mv) } \psi(\pi^{-\lv}  \niw(x)  \pi^{\lv} ) \ee where again we must use the fact that $\lv \in \Lv_+$ to note that $\psi(\pi^{-\lv}  \niw(x)  \pi^{\lv} )$ is well-defined for $x \in M(0; \mv).$

\end{enumerate} 

\subsection*{Gindikin-Karpelevic Limits} \noindent

\spoint The \emph{Gindikin-Karpelevic} sum is defined as \be{gksum} \mc{G}:= \sum_{\mv \in \Lv} | K \setminus K \um \cap K \pi^{\mv} U | e^{\mv} q^{\la \rho, \mv \ra } = \sum_{\mv \in \Lv} | M(0; \mv)| e^{\mv} q^{ \la \rho, \mv \ra } . \ee  \begin{nthm}\cite[Theorem 1.13]{bgkp}  The expression $\mc{G}$ is a well-defined expression in $\C_{\leq}[\Lv]$ and  \label{gk} \be{gk-eval} \mc{G} = \mf{m} \, \Delta \ee where  $\mf{m}$ is as in (\ref{m}) and \be{Del} \Delta = \prod_{ a \in R_+} (\frac{ 1- q^{-1} e^{-\av}}{1- e^{-\av}})^{m(\av)}, \ee   \end{nthm} 

To relate the Gindikin-Karpelevic sum $\mc{G}$ and the Whittaker sum $\W(\pi^{\lv})$, we introduce the following definition. 

\begin{de} Let $\mv \in \Lv$ and $\lv \in \Lv_+.$ 
\begin{enumerate} 
\item We shall say that $\mv$ is of Gindikin-Karpelevic type for $\lv$ if \be{gk-typ} [e^{\mv - \lv}] \W(\pi^{\lv}) = q^{ \la \rho, \lv \ra} [e^{\mv}] \mc{G}. \ee 
\item  We shall say that $\lv$ is sufficiently dominant compared to $\mv$ if $\la a_i, \lv + \mv \ra$ is sufficiently large for each $i=1, \ldots, \ell+1.$ 
\end{enumerate}
\end{de} 

With this definition, we can now state,

\begin{nprop} \label{gk-whit} Fix $\mv \in \Lv.$ If $\lv \in \Lv_+$ is sufficiently dominant (compared to $\mv$), then  $\mv$ is of Gindikin-Karpelevic type for $\lv.$  \end{nprop}

\begin{proof} Let us actually show the Proposition for a fixed $-\mv \in \Lv$ (and $\lv \in \Lv_+$ to be chosen sufficiently dominant to $-\mv$). Choose any $\nm \in \um \cap K \pi^{- \mv} \up.$ In light of (\ref{W-alt}), it suffices to show that $\pi^{-\lv} \nm \pi^{\lv} \in \um_{\O}.$ Equivalently, if we write $\nm = \pi^{\lv} \wt{n}^- \pi^{-\lv}$ with $\wt{n}^- \in \um,$ then we need to show that $\wt{n}^- \in \um_{\O}.$  Now, our assumption that $\nm = \pi^{\lv} \wt{n}^- \pi^{-\lv} \in K \pi^{- \mv} \up$ implies that $\pi^{\lv} \wt{n}^- \in K \pi^{\lv - \mv} \up.$ From \cite[Theorem 1.9]{bgkp}: for fixed $\mv$ and $\lv$ sufficiently dominant relative to $\mv$, we have \be{bau-gau-approx} K \pi^{\lv} \um \cap K \pi^{\lv - \mv} \up = K \pi^{\lv} \um_{\O} \cap K \pi^{\lv - \mv} \up. \ee So we may write $\wt{n} = z v$ with $v \in \um_{\O}$ and $\pi^{\lv} z \pi^{-\lv} \in \um_{\O}.$ As $\lv \in \Lv_+$ the condition $\pi^{\lv} z \pi^{-\lv} \in \um_{\O}$ implies that $z \in \pi^{-\lv} \um_{\O} \pi^{\lv} \subset \um_{\O}.$ Hence $\wt{n}^- \in \um_{\O},$ and so $\pi^{-\lv} \nm \pi^{\lv} = \wt{n}^- \in \um_{\O}$ which is what we wanted to show.

\end{proof} 

\subsection*{Iwahori-Whittaker Sums} \label{sec-iwit} \noindent

\renewcommand{\f}{\mathbf{f}}

\spoint Let $w \in W$ and $\lv \in \Lv_+.$ Consider the following maps induced by multplication, \be{ref:1} m_{\lv} &:& \up K \times_K K \pi^{\lv} \um \rr G \\ m_{w, \lv} &:& \up w I^- \times_{I^-} I^- \pi^{\lv} \um \rr G. \ee For $\mv \in \Lv$ we consider the fibers $m_{\lv}^{-1}(\pi^{\mv})$ and $m_{w, \lv}^{-1}(\pi^{\mv}).$ From \cite[\S 4.1.3]{bkp} we know that \be{K-fib}  m_{\lv}^{-1}(\pi^{\mv}) = M(\lv; \mv) = K \setminus K \pi^{\lv} \um \cap K \pi^{\mv} \up, \ee the set involved in the definition of the Whittaker sums. In analogy with \cite[Lemma 7.3.1]{bkp}, we shall now argue the following.

\begin{nlem} \label{iw-k:lem} For a fixed $\lv \in \Lv_+$ and $\mv \in \Lv$, the fibers $m_{w, \lv}^{-1}(\pi^{\mv})$ are disjoint for $w \in W$ and moreover we have a bijective correspondence 
\be{iw-k-sums} \sqcup_{w \in W} m_{w, \lv}^{-1}(\pi^{\mv}) = m_{\lv}^{-1}(\pi^{\mv}) \ee  
\end{nlem}

\begin{proof} Consider the map \be{varphi} \varphi: \sqcup_{w \in W} m_{w, \lv}^{-1}(\pi^{\mv}) \rr m_{\lv}^{-1}(\pi^{\mv}) \ee which sends $(a, c) \mapsto (a, c)$ where $a \in \up w I^-$ and $c \in I^- \pi^{\lv} \um.$ This map is well-defined: if $j \in I^-$ and $(a, c) \in m_{w,\lv}^{-1}(\pi^{\mv})$ then $(aj, j^{-1}c)$ represents the same element in $m_{\lv}^{-1}(\pi^{\mv}).$ Denote by $\varphi_w$ the restriction of $\varphi$ to each $m_{w, \lv}^{-1}(\pi^{\mv}).$ We claim that $\varphi$ is the required bijection. 

That $\varphi$ is surjective follows at once from the Iwahori-Matsumoto decomposition \be{Kim} K = \sqcup_{w \in W} I^- \wt{w} I^-, \ee and the decomposition of $I^-=\up_{\pi} \um_{\O} A_{\O}= \um_{\O} \up_{\pi} A_{\O}$ as described in \cite[Proposition 3.12]{bkp}.

\noindent To prove the injectivity of $\varphi$ we need to show
\begin{description}
\item[(i)] the images of $\varphi_w$ and $\varphi_{w'}$ are disjoint for distinct elements $w, w' \in W$; and 
\item[(ii)] the map $\varphi_w$ is injective for $w \in W$ \end{description} 

\begin{nclaim} Let $(a, c), \, (a', c') \in \sqcup_{w \in W} m_{w, \lv}^{-1}(\pi^{\mv})$ be such that $\varphi(a, c) = \varphi(a', c').$ Then there exists (a unique) $w \in W$ such that $(a, c), \, (a', c') \in m_{w, \lv}^{-1}(\pi^{\mv}).$  
\end{nclaim}
\begin{proof} Since $\varphi(a, c) = \varphi(a', c'),$ by definition there exists $v \in K$ such that $c = v c'.$ Suppose that $v \in I^- \sigma I^-,$ with $\sigma \in W.$ The claim will follow if we can show that $\sigma=1.$ If $c = v c',$ then by the assumption we have made on $v$ we have that \be{cl:1} I^- \pi^{\lv} \um \cap I^- \sigma I^- \pi^{\lv} \um = I^- \pi^{\lv} \um \cap I^- \sigma \pi^{\lv} \um \neq \emptyset .\ee By the disjointness of the decomposition (see \cite[\S 3.4]{bkp}), we may conclude that $\sigma=1$ as desired. 
\end{proof}
The statement (i) above is just the above Claim. As for (ii): if $(a, c)$ and $(a', c')$ belong to $m_{w, \lv}^{-1}(\pi^{\mv})$ and have the same image by $\varphi$ then we must again have $c' \in I^- c.$ Hence $(a, c)$ and $(a', c')$ are equivalent in $m_{w, \lv}^{-1}(\pi^{\mv}).$ \end{proof}

\spoint For fixed $\lv \in \Lv_+$, the restriction of (\ref{niw:def-a}) to $I^- \pi^{\lv} \um$ gives a map \be{proj:nm} \niw: I \pi^{\lv} \um \rr  \um_{\O} \setminus \um, \ee and hence a map $\niw: m_{w, \lv}^{-1}(\pi^{\mv}) \rr \um_{\O} \setminus \um$ by composing the projection of $m_{w, \lv}^{-1}(\pi^{\mv})$ onto its last coordinate.  

\begin{de} \label{iw-whit:defn} Let $w \in W, \lv \in \Lv_+,$ and $\psi$ a principal, unramified character. Then the \emph{Iwahori-Whittaker} sum $\W_{w, \lv}$ is defined as, \be{iw-whit}
 \W_{w, \lv} &:=& \sum_{\mv \in \Lv} e^{\mv} q^{\la \rho, \mv \ra} \sum_{ x \in m_{w, \lv}^{-1}(\pi^{\mv})}  \psi(\niw(x)). \ee \end{de}

\rpoint{Remarks} \begin{enumerate}
\item  In light of Lemma \ref{iw-k:lem}, and (\ref{K-fib}), for any $\lv \in \Lv_+$ we have \be{W:sum-w} \W(\pi^{\lv}) = \sum_{w \in W} \W_{w, \lv}. \ee

\item The $\W_{w, \lv}$ should be compared to the $\W_{o}(\pi^{\lv})$ from (\ref{W:wpart}). We shall actually view $\W_{w, \lv}$ as a function of $G$ below (see Lemma \ref{whit:as-av}) and for this reason we prefer to adopt this peculiar notation. 

\item In \S \ref{fin_2}, we show how the recursion relation Proposition \ref{recur-body} it follows that the sum $\W_{w, \lv}$ has finite support. It would be interesting to see this more directly from the definitions. 

\end{enumerate}

\section{Recursion formulas} \label{sec-rec}
\subsection*{Function Spaces and Some Operators } \noindent
\newcommand{\e}{\mathbf{e}}

\spoint We shall generally be interested in the space $M(G)$ of functions on $G$ which are left-invariant under the group $A_{\O}U.$ Such functions carry a natural action of $\C[\Lv]$ on the left, defined by requiring that  \be{act:2} e^{\lv} f(g) = q^{ \la \rho, \lv \ra} f(\pi^{-\lv} g) \text{ for } g \in G, \lv \in \Lv, \ee and then extending linearly to $\C[\Lv].$

Recall that the group $G$ may be decomposed into disjoint $(A_{\O} U, I^-)$ double cosets parametrized by the "affinized" Weyl group $\aw:= \Lv \rtimes W,$ a typical element of which will be written as $x=(\mv, w) \in \aw$ where $\mv \in \Lv, w \in W.$  Hence, any function $f: G \rr \C$ which is left $A_{\O}U$ and right $I^-$ invariant may be written as a (possibly infinite) sum of the following form \be{f:exp} f = \sum_{x \in \aw } c_x \v_x, \, c_x \in \C, \ee and where $\v_x$ is the characteristic function of the double coset $A_{\O} \up x I^-, x \in \aw.$ Note that here we are using the symbol $x$ to denote both the element in the abstract "affine" Weyl group and for any lift of it to $N \subset G.$ Let \be{M:i} M(G, I^-) = \C(A_{\O}\up \setminus G / I^-) \ee denote the set of all such functions. Under the action (\ref{act:2}), we have \be{act:base} e^{\lv} \v_x = q^{ \la \rho, \lv \ra}  \v_{\lv. x}, \ee where if $x = (\mv, w) \in \aw$ as above, then $\lv. x = (\mv+\lv, w) \in \aw$. The space $M(G, I^-)$ also carries a right action by convolution with the Hecke algebra $H^-_W:= \C(I^- \setminus K / I^-)$ defined in \cite[\S5.4]{bkp}. The algebra $H^-_W$ has a basis $T_w,$ $w \in W$ consisting of the characteristic functions of the double cosets $I^- w I^-.$ We denote this right action by $\star$ and note that \be{star:act} \ve_{\lv} \star T_w &=& \ve_{(\lv, w)}, \, \text{ where }   w \in W,  \lv \in \Lv. \ee

\spoint Let $\psi: \um \rr \C^*$ be any character, and let $M(G, \um, \psi)$ denote the space of functions on $G$ which are $A_{\O} \up$-left invariant and satisfy the following invariance on the right under $\um:$ \be{f-rt} f(g \nm) = \psi(\nm) f(g). \ee  Such functions carry a natural action of $\C[\Lv]$ on the left using the same formula as in (\ref{act:2}).  

Consider the element $ \wid \in M(G, \um, \psi)$ defined as follows, \be{wid-def} \wid(g) = \begin{cases} \psi(\nm) & \text{ if } g = n n^- \text{ where } n \in \up, n^- \in \um \\  0 & \text{ otherwise} \end{cases} \ee  The support of such a function is contained in the big cell $B \um.$ We denote by $M^{\gen}(G, \um, \psi)$ all such functions in $M(G, \um, \psi)$ which are supported on $B\um,$ i.e., those functions $f \in M(G, \um, \psi)$ of the form \be{f:gen} f = \sum_{\mv \in \Lv} c_{\mv} e^{\mv} \wid, \text{ where } c_{\mv} \in \C. \ee Note that  $c_{\mv}= f(\pi^{\mv})q^{ -\la  \rho, \mv \ra}.$

\tpoint{Averaging Operators} Let $\psi: \um \rr \C^*$ now be an unramified, principal character as in \S \ref{chars}, and let $\lv \in \Lv_+.$ Then the two function spaces $M(G, I^-)$ and $M(G, \um, \psi)$ may be connected using an averaging operator \be{avg:1} \Av_{\psi, \lv}: M(G, I^-)  \rr M(G, \um, \psi),  \ee defined as follows. If $x \in \aw$ and $\ve_x = \mathbf{1}_{A_{\O} \up x I^-}$ the characteristic function of $A_\O U xI^-$, we set  \be{avg:2} \Av_{\psi, \lv} (\ve_x) (g) = \sum_{z \in m_{x, \lv}^{-1} (g) } \psi(\niw(z)) \text{ for } g \in G,\ee where $m_{x, \lv}: \ve_x \times_{I^-} I^- \pi^{\lv} \um \rr G$ is the usual multiplication map and $\niw: m_{x, \lv}^{-1} (g) \rr \um_{\O}  \setminus \um$ is the projection defined as in (\ref{proj:nm}). We note that $\Av_{\psi, \lv}(\ve_x)$ is left $\up$-invariant and is right $(\um, \psi)$-invariant (i.e., satisfies (\ref{K:f:n}) ). However, the function $\Av_{\psi, \lv}(\ve_x)$ is not necessarily in $M^{\gen}(G, \um, \psi),$ i.e., its support may lie outside of the big cell. We let $\Av^{\gen}_{\psi, \lv}$ be the restriction of $\Av_{\psi, \lv}$ to the big cell, i.e., \be{av-def} \Av^{\gen}_{\psi, \lv}(\ve_x)(g) = \begin{cases} \Av_{\psi, \lv}(\ve_x)(g) & \text{ if } g \in B \um \\
0 & \text{ otherwise} . \end{cases} \ee We may extend the above construction linearly to define maps \be{av-maps} \begin{array}{lcr} \Av_{\psi, \lv}: M(G, I^-)  \rr M(G, \um, \psi) \text{ and } \Av^{\gen}_{\psi, \lv}: M(G, I^-)  \rr M^{\gen}(G, \um, \psi) \end{array}. \ee We record here a simple property: for  $f \in M(G, I^-)$ and $\mv \in \Lv$ \be{avg-prop:1} \Av_{\psi, \lv}( e^{\mv} f) = e^{\mv} \Av_{\psi, \lv}(f).\ee To prove this formula, one may reduce to the case that $f= \ve_w$ for $w \in W,$ in which case the result follows from the definitions. Finally, note the following, which is essentially just a consequence of the definitions and which establishes the connection between the averaging operators and the Iwahori Whittaker functions defined in (\ref{iw-whit:defn}). 

\begin{nlem} \label{whit:as-av} For $w \in W$ and $\lv \in \Lv_+.$ Let $\psi$ be a principal, unramified character. Then we have \be{avg:whit:1} \W_{w, \lv} = \sum_{\mv \in \Lv} e^{\mv} q^{ \la \rho, \mv \ra} \Av_{\psi, \lv}(\ve_w)(\pi^{\mv}) \ee \end{nlem} 

\rpoint{Notation} For a function $f \in \C(A_{\O} \up \setminus G)$ we define the formal series $\Psi(f)$ which is constructed from the values of $f$ on the torus as follows,  \be{Psi:f} \Psi(f):= \sum_{\mv \in \Lv} f(\pi^{\mv}) e^{\mv} q^{\la \rho, \mv \ra }. \ee In this notation, the above Lemma states, \be{whit:av-2} \Psi(\Av_{\psi, \lv}(\ve_w) ) = \W_{w, \lv} \ee

\tpoint{Intertwining Operators} For each $a \in \Pi$ we may define the operator \be{I_a} I_a: M(G) \rr M(G) \ee via the following integral formula: let $\varphi \in M(G)$ and $x \in G,$ \be{int:def} I_a(\varphi) (x)  = \int_{U_a} \varphi(w_a u_a x) du_a, \ee where $du_a$ is the Haar measure on $U_a$ which gives $U_a \cap K$ volume $1.$ It is well-known that the operators $I_a$ satisfy the braid relations, i.e. if $w \in W$ has a reduced decomposition $w= w_{a_1} \cdots w_{a_r}$ with $a_i \in \Pi$ for $i=1, \ldots, r,$ and we define \be{Iw} I_w:= I_{a_1}  \circ \cdots \circ I_{a_r}, \ee then $I_w$ does not depend on the reduced decomposition chosen. One may also verify the following standard facts, \be{int-flas} I_w ( e^{\lv} f) &=& e^{w \lv} I_a(f) \text{ for } w \in W,\, \lv \in \Lv, \,  f \in M(G) \\
I_w(f \star h) &=& I_w(f) \star h \text{ for } w \in W, \, h \in H^-_W,\, f \in M(G, I^-) 
\ee 

The following result summarizes the properties of the operators introduced above which we shall need for the proof of Proposition \ref{recur-body}.

\begin{nlem} \label{it-sum} Let $\psi: \um \rr \C$ be an unramified, principal character constructed as in \S \ref{chars} and $\lv \in \Lv_+.$ Then 
\begin{enumerate} 
\item $I_w( \Av_{\psi, \lv}(f) ) = \Av_{\psi, \lv}(I_w( f))$ for $w \in W$ and $f \in M(G, I^-).$
\item $I_a (\ve_1) = \ve_{w_a} + \frac{1 - q^{-1}}{1- e^{\av}} \ve_1= \ve_{w_a} - \b(\av)\ve_1$
\item  $I_a (\wid) = \frac{1 - q^{-1} e^{-\av}}{1 - e^{\av}} \wid = \c(\av) \wid , $

\end{enumerate} where the rational expressions in parts (2) and (3) are expanded in positive powers of $e^{\av}$ and we have used the same notation as in the finite dimensional case for $\c(\av)$ and $\b(\av)$ (see \ref{b:c}). \end{nlem}

The first part of the theorem follows from a Fubini-type result in our algebraic setting. The second and third parts are simple rank 1 computations.

\subsection*{Iwahori-Whittaker Recursion Formula} \noindent

\spoint \label{completions} Let $J \subset R_{re}$ be a finite subset of real roots, and consider the subalgebra of $B_J = \C [ e^{a^{\vee}} ]_{a \in J} \subset \C[\Lv].$  Denote by $\hat{B}_J$ the completion of $B_J$ with respect to the maximal ideal spanned by $e^{a^{\vee}}, a \in J.$    Let us set \be{cj} \C_J[\Lv] :=  \hat{B}_J \otimes_{B_J} \C[\Lv]. \ee  For each $w \in W$ we let $S_w = R_{+} \cap w R_{-}$ and write $\C_{w}[\Lv]$ in place of $\C_{S_w}[\Lv].$ The following is the main recursion result we wish to show,

\begin{nprop} \label{recur-body} Fix $\lv \in \Lv$ and let $w, w' \in W$ be related as follows: $w = w_a w'$ with $a \in \Pi_o$ a simple root, and $\ell(w) = 1 + \ell(w').$ Then we have the following identity in the completion, $\C_{w_a}[\Lv],$ \be{W:rec} \W_{w,\lv} = \c(\av) \W_{w', \lv}^{w_a} + \b(\av) \W_{w', \lv}, \ee where $\W_{w', \lv}^{w_a}$ denotes the termwise application of $w_a$ to the expression $\W_{w', \lv},$ and where $\b(\av)$ and $\c(\av)$ are as in (\ref{b:c}). \end{nprop}

\rpoint{Remark}  

\begin{enumerate}

\item If $w=1$ then we have that $\W_{1, \lv} = q^{ \la \rho, \lv \ra} e^{\lv}$ since $K \pi^{\lv} \up \cap K \pi^{\lv} \um = K \pi^{\lv}$ (see for example Theorem \ref{gk}). 
\item In fact, the above equality holds in $\C[\Lv],$ i.e., expanding the rational functions in the right hand side of (\ref{W:rec}) in $\C_{w_a}[\Lv],$ a cancellation occurs to ensure the whole expression lies in $\C[\Lv]$.

\end{enumerate}

\begin{proof}[Proof of Proposition ] \label{pf-recur} 

Let $w \in W$ be written as $w = w_a w'$ as in the statement of the Proposition. From Lemma \ref{it-sum} (1), we have an equality of functions on $G$, \be{rec:1} I_a ( \Av_{\psi, \lv}(\ve_{w'}) ) =  \Av_{\psi, \lv}(I_a (\ve_{w'} )). \ee
The right hand side is computed using Lemma \ref{it-sum}(3) and (\ref{int-flas}), \be{rec:2} \Av_{\psi, \lv}(I_a (\ve_{w'} )) = \Av_{\psi, \lv}(\ve_{w_a w'} - \b(\av) \ve_{w'} ) =   \Av_{\psi, \lv}(\ve_{w_a w'} ) -  \b(\av) \Av_{\psi, \lv}(\ve_{w'}). \ee Now we focus on the left hand side of (\ref{rec:1}). We begin by noting,

\begin{nclaim} On the big cell, we have \be{i-gen:0} I_a ( \Av_{\psi, \lv}(\ve_{w'}) ) (g)  = I_a ( \Av^{\gen}_{\psi, \lv}(\ve_{w'}) ) (g) \text{ for } g \in B \um.\ee \end{nclaim} 

\begin{proof}[Proof of Claim] It suffices, in view of the left $A_{\O} \up$ and right $(\um, \psi)$ invariance of both sides to verify the statement for $g = \pi^{\mv}, \mv \in \Lv.$ To do so, we compute \be{i-gen:1} I_a ( \Av_{\psi, \lv}(\ve_{w'}) ) (\pi^{\mv}) = \int_{U_a} \Av_{\psi, \lv}(\ve_{w'}) (w_a n_a \pi^{\mv}) du_a =  \int_{U_a \setminus \{ 1 \} } \Av_{\psi, \lv}(\ve_{w'}) (w_a n_a \pi^{\mv}) du_a
\ee since the set $\{ 1 \} \subset U_a$ has measure 0. On the other hand, if $n_a \in U_a \setminus \{ 1 \}$ then $w_a n_a \pi^{\mv} \in B\um$ and so, in the last integral, we may replace the function being integrated to $\Av^{\gen}_{ \psi, \lv}(\ve_{w'}).$ \end{proof}

Since $\Av^{\gen}_{\psi, \lv} \in M^{\gen}(G, \um, \psi),$ as in (\ref{f:gen}) we may write \be{rec:4} \Av^{\gen}_{\psi, \lv}(\ve_w') = \sum_{ \xv \in \Lv} c_{\xv} e^{\xv} \wid, \ee with (see the line after (\ref{f:gen})), \be{c:xv} c_{\xv} = q^{ - \la \rho, \xv \ra} \Av^{\gen}_{\psi, \lv}(\ve_w')(\pi^{\xv}). \ee Using Lemma \ref{it-sum}(3) , (\ref{int-flas}) and (\ref{i-gen:0}), for $g$ in the big cell we have \be{rec:4} I_a ( \Av_{\psi, \lv}(\ve_{w'}) ) (g) = I_a ( \Av^{\gen}_{\psi, \lv}(\ve_{w'}) ) (g) =   \c(\av) \sum_{\xv \in \Lv} c_{\xv} e^{w_a \xv} \wid (g) \ee Now, from Lemma \ref{whit:as-av} (and in the notation directly following it), (\ref{W:rec}) may be restated as
\be{p-rec:0} \Psi( \, \Av_{\psi, \lv}(\ve_{w})) = \b(\av) \Psi( \Av_{\psi, \lv}(\ve_{w'}) \,) + \c(\av)  \Psi(\, \Av_{\psi, \lv}(\ve_{w'}) \,)^{w_a}, \ee which is the same as  showing
 \be{p-rec:1} \Psi( \, \Av_{\psi, \lv}(v_{w_a w'}) - \b(\av) \Av_{\psi, \lv}(v_{w'}) \,) = \c(\av) \sum_{\mv \in \Lv} \Av_{\psi, \lv}(\ve_{w'})(\pi^{\mv}) e^{w_a \mv} q^{ \la \rho, \mv \ra}. \ee By (\ref{rec:2}), the left hand side can be interpreted as $\Psi(\Av_{\psi, \lv}(I_a(v_{w'}))).$ As for the right hand side, it may be written as \be{p-rec:4} \c(\av) \sum_{\mv \in \Lv}  e^{w_a \mv} q^{ \la \rho, \mv \ra}  \sum_{\xv \in \Lv} c_{\xv} e^{\xv} \wid (\pi^{\mv}) &=& \c(\av) \sum_{\mv \in \Lv} e^{w_a \mv} q^{ \la \rho, \mv \ra} c_{\mv} \wid(\pi^{\mv}) \ee which, from (\ref{rec:4}) is equal to $\Psi( I_a ( \Av_{\psi, \lv}(\ve_{w'}) ) ).$  Our Proposition (i.e., (\ref{p-rec:0})) now follows from (\ref{rec:1}). \end{proof}

\section{Algebraic Identities} \label{sec-alg} 

\subsection*{Hecke Algebras and some Demazure-Lusztig Operators} \noindent

\spoint Let $v$ be a formal variable, and let us denote the ring of rational functions on the variables $e^{\lv}, \lv \in \Lv$ with coefficients in $\C[v^{-1}]$ by $R:= \C[v^{-1}](\Lv).$ The following two special rational functions will be important for us \footnote{They were already introduced earlier under the same name in the special case $v=q$, we hope this will cause no confusion here}: for $a \in \Pi$ set,  \be{rat:a} \c(\av) := \frac{ 1 - v^{-1} e^{-\av}}{1 - e^{\av}}  & \text{ and } & \b(\av) :=  \frac{v^{-1} -1}{1 - e^{\av} }. \ee  Let $R[W]$ be the group algebra of $W$ over $R,$ and for each $a \in \Pi$ consider the following element in $R[W]$ \be{Ta} T_a := \c(\av) [w_a] + \b(\av) [1]. \ee  

\begin{nprop} \label{quad:braid} The elements $T_{a}$ verify the following two properties, 
\begin{enumerate} 
\item $T_a^2 = (v^{-1}-1) T_a + v^{-1}$ for each $a \in \Pi$
\item The expressions $T_a$ verify the braid relations, and so we may write unambiguously \be{Tw} T_w := T_{w_1} \cdots T_{w_r} \ee where $w = w_1 \cdots w_r$ is \emph{any} reduced decomposition of $w.$
\end{enumerate} \end{nprop} 
\begin{proof} The first part is a simple verification. One may verify the second part directly, or proceed as follows. For each $a \in \Pi$, define the rational expressions in $R'=\C[v, v^{-1}](\Lv)$ \be{tc} \begin{array}{lcr} \c'(\av) = \frac{ 1 - v e^{\av}}{1 - e^{\av} }& \text{ and } & \b'(\av) = \frac{v -1}{1 - e^{\av}} \end{array} \ee and the corresponding element in $R'[W]$  \be{T'} T'_a = \c'(\av) [w_a] + \b'(\av) [1]. \ee It is known that $T'_a$ satisfy the braid relations (see \cite{lus}). On the other hand, one can show easily that as elements of $R'[W]$ we have for each $a \in \Pi$ \be{T:T'} e^{- \rho} \; T'_a \;  e^{\rho} = - v T_a .\ee The braid relations for $T_a$ follow from those of $T'_a$.  \end{proof} 

\rpoint{Remarks}
\begin{enumerate}
\item   The algebra $\mathbf{H}_v$ generated by $T_a$ subject to the above relations is of course just the Hecke algebra of $W.$ 
\item For each $w \in W$ and $\lv \in \Lv$ we have $T_w(e^{\lv}) \in \C[v][\Lv]$ if we expand all the rational functions involved in the expression $T_w(e^{\lv}).$ This is known for the operators $T'_w$ defined using \ref{T'} (see \cite[\S4.3]{mac:aff}); this representation is sometimes known as Cherednik's polynomial representation, where they played a prominent role in the double affine setting (note though this representation also appeared earlier, for example in Lusztig's geometric construction of the affine Hecke algebra) . The statement for $T_w$ follows using the formula (\ref{T:T'}).  
\end{enumerate}

\subsection*{The operator $\mc{P}$} \noindent

\spoint We shall need the following completions (in the variables $e^{-\av}, \av \in Q^{\vee}_+$) for the sequel \be{qn} \begin{array}{lcr} \qn:= \C[[v^{-1}]] [[Q^{\vee}_-]] & \text{ and } & \qnf:= \C[v^{-1}][[Q^{\vee}_-]] \end{array}. \ee The latter will sometimes be referred to as the set of $v$-finite elements in the former. We also let $\qn[W]$ and $\qn[W]^{\vee}$ denote respectively the group algebra over $\qn$ and the vector space of (possibly infinite) formal linear combinations \be{gen:elt} f:= \sum_{w \in W} a_w [w] \, \text{ where } a_w \in \qn. \ee Similarly we may define $\qnf[W]$ and $\qnf[W]^{\vee}.$

\spoint Each of the elements $T_a, \, a \in \Pi$ defined in (\ref{Ta}) may be regarded as an element in $\qnf[W]$ by formally expanding the rational functions $\c(\av)$ and $\b(\av)$ in negative powers of the coroots. Similarly, we may regard $T_w \in \qnf[W]$ by first using (\ref{Ta}) and (\ref{Tw}) to write $T_w$ as an element in $R[W]$ and then expanding the rational functions which appear in negative powers of the coroots. 

We would like to explain precisely the meaning of the following formal sum \be{sym} 
\mc{P}:= \sum_{w \in W} T_w . \ee   On the one hand, it is easy to see that \be{Tw} T_w = \sum_{u \leq w} C_u(w) [u]  \in \qnf[W] \ee where $u \leq w$ denotes elements in $W$ which are less than $w$ in the Bruhat order. On the other hand, considering the sum over all such $T_w$, it can certainly be the case that for any fixed $u \in W$ there are infinitely $w \in W$ such that $C_u(w) \neq 0.$ Let us formally collect coefficients and write \be{P:for:1} \mc{P} = \sum_{u \in W} (\sum_{w \in W} C_{u}(w) ) [u], \ee and then set \be{C_u} C_u :=  \sum_{w \in W} C_{u}(w). \ee Following Cherednik \cite[Lemma 2.19]{cher:ma}, one may show that $C_u$ has a well-defined expansion, not just in $\qn$, but actually

\begin{nprop} \label{sym:wd} For each $u \in W$ we have $C_u \in \qnf.$ Hence the element $\mc{P}$ lies in $\qnf[W]^{\vee}.$ \end{nprop} \noindent Note that  the proof in \emph{op. cit} is written for the slightly different operator \be{cher-op} \wt{T}_a= \wt{c}(\av)[w_a] + \wt{\b}(\av)[1] \ee where $\wt{c}(\av)$ and $\wt{b}(\av)$ are as in (\ref{tc}) with $v$ replaced by $v^{-1}$. However, the argument holds for our operators $T_a$ since it relies just on the expansion of $\wt{\b}(\av)= \b(\av).$

\subsection*{An "Operator" Casselman-Shalika Formula} \noindent

\newcommand{\bv}{b^{\vee}}

\spoint  In analogy to the definition of $\mf{D}$ given in (\ref{del}) we now define its $v$-deformation, \be{del:q}  \mf{D}_v := \prod_{ a \in R_+} (1 - v^{-1} e^{-\av})^{m(\av)} \in \qnf \ee where $m(\av)$ is the multiplicity of the root $\av.$  

\begin{nlem} \label{aux:algid} Let $a \in \Pi.$ Then we have the following equalities in $\qnf[W]^{\vee}$ \begin{enumerate} 
\item[(i)] $T_a \mc{P}  =  v^{-1} \mc{P}$ 
\item[(ii)] $\mc{P} T_a = v^{-1} \mc{P}$ 
\item[(iii)]  $w_a \mc{P} = \c(-\av) \mc{P}$ 
\item[(iv)] The quantity $\frac{1}{\mf{D}_v} \mc{P}$ is $W$-invariant. 
\end{enumerate} \end{nlem} 
\rpoint{Remark} Implicit in each of the above statements is that the quantities in the left hand side of (i), (ii), (iii) are well-defined. We refer the reader to \cite[\S 7.3.8]{bkp} for more explanations on this point.

\begin{proof} The proof of (i) and (ii) is similar to \cite[5.5.9]{mac:aff} and so we suppress the details. As for (iii), recalling that $\c(\av) [w_a] + \b(a) [1] = T_a$ we may write formally $[w_a] = \c(\av)^{-1}(T_a - \b(a))$ and so using part (i), \be{lem:2-2} w_a \mc{P} = \frac{1}{\c(\av)} ( v^{-1} \mc{P} - \b(\av) \mc{P} ) = \frac{v^{-1} - \b(\av)}{\c(\av)} \mc{P}. \ee We leave it to the reader to verify that $\frac{v^{-1} - \b(\av)}{\c(\av)} = \c(-\av).$ As for part (iv), it suffices to show for each $a \in \Pi$ that \be{wa:inv} w_a \frac{1}{\mf{D}_v} \mc{P} = \frac{1}{\mf{D}_v} \mc{P}. \ee  The left hand side of (\ref{wa:inv}) is equal to \be{lhs:lem3} \prod_{ b \in R_+, b \neq a} (1 - v^{-1} e^{-\bv})^{-m(\bv)} \ \frac{1}{1 - v^{-1} e^{\av}} \  w_a \mc{P}. \ee The result follows from the part (iii). \end{proof} 

\spoint The following result should be compared to (\ref{alg:cs}) from the finite-dimensional setting.

\begin{nprop} \label{alg-iden} There exists an element $\mf{c}(v) \in \C[v^{-1}][[e^{-\cc}]] \subset \qn$ such that we have the following equality of elements in $\qnf[W]^{\vee},$ \be{alg-id:1} \mc{P} = \mf{c}(v) \, \mf{D}_v \sum_{w \in W} [w] \frac{1}{\mf{D}}. \ee \end{nprop}  

\begin{proof} First note that from Proposition \ref{sym:wd} and Lemma \ref{aux:algid},  there exists $\Gamma \in \qnf$ such that we may write \be{p:gam} \frac{1}{\mf{D}_v} \mc{P} = \sum_{w \in W} [w] \Gamma. \ee Now from Lemma \ref{aux:algid} (ii), we have \be{alg:1} \frac{1}{\mf{D}_v} \mc{P} T_a = \frac{1}{\mf{D}_v} \mc{P} v^{-1} \ee and so \be{alg:2} \sum_{w \in W} [w] \Gamma ( \c(\av)[w_a] + \b(\av)[1]) = \sum_{w \in W} [w] \Gamma v^{-1}. \ee  Hence, \be{alg:3} \sum_{w \in W} [w w_a] \Gamma^{w_a} \c(-\av) + \sum_{w \in W} [w] \Gamma \b(\av) &=& \sum_{w \in W} [w] \Gamma v^{-1} \ee which is only the case if \be{alg:4}  \Gamma^{w_a} \c(-\av) + \Gamma \b(\av) =  \Gamma v^{-1}, \ee or equivalently \be{alg:4} \Gamma^{w_a} \c(-\av) = \Gamma ( v^{-1} - \b(\av)). \ee Hence  \be{alg:5} \frac{\Gamma^{w_a}}{\Gamma} = \frac{v^{-1} - \b(\av)}{\c(-\av)} = \frac{1 - e^{-\av}}{1 - e^{\av}} . \ee  On the other hand, we may verify that \be{alg:6} \frac{\mf{D}}{\mf{D}^{w_a}} = \frac{1 - e^{-\av}}{1 - e^{\av}} \ee as well. Thus we conclude that $\mf{c}(v):= \frac{\Gamma}{\mf{D}^{-1}} \in \qnf$ is $W$-invariant. Using an argument as in \cite{mac:formal} we conclude that $\mf{c}(v) \in \C[v^{-1}][[e^{-\cc}]].$ \end{proof}

\section{Proof of Main Theorem} \label{sec-pf}

\spoint Let $\mf{m}_v$ be defined by the same formula as (\ref{m}) but with all $q$'s replaced by a formal variable $v.$ The main result of this paper is then the following,

\begin{nthm} \label{main} Let $\lv \in \Lv_+.$ Then we have \be{} \mf{m}_v \, \prod_{a \in R_+ } (1 - v^{-1} e^{-\av})^{m(\av)} \chi_{\lv} \in \C[v^{-1}]_{\leq}[\Lv] \ee and its value at $v=q$ is equal to $\W(\pi^{\lv}),$ i.e., \be{cs-body} \W(\pi^{\lv}) = \mf{m} \, \prod_{a \in R_+ } (1 - q^{-1} e^{-\av})^{m(\av)} \chi_{\lv} \ee \end{nthm}    

The remainder of this section is devoted to the proof of this result. 

\spoint \label{fin_2} We have already observed in (\ref{W:sum-w}) that $\W(\pi^{\lv}) = \sum_{w \in W} \W_{w, \lv}.$ Using the recursion result Proposition \ref{recur-body}, an argument as in \cite[Prop 7.3.3]{bkp} shows that
that \be{W:t} \W_{w, \lv} = q^{ \la \rho, \lv \ra}T_w(e^{\lv}), \ee where the right hand side is expanded in \emph{positive powers of the coroots} and then evaluated at $v=q.$ On the the other hand, from remark (2) after Proposition \ref{quad:braid}, $T_w(e^{\lv}) \in \C[\Lv]$ so regardless of whether we expand the rational functions appearing in $T_w(e^{\lv})$ as positive or negative powers of the coroots, we obtain the same quantity. To remain within the ring $\qnf$, we shall expand in \emph{negative} powers of the coroots from now on. Note that from (\ref{W:t}), we may conclude that $\W_{w, \lv} \in \C[\Lv],$ i.e., it has finite support.  Moreover, as in \emph{loc. cit}, we may argue that there exist polynomials $\Upsilon_{w, \mv}(v^{-1}) \in \C[v^{-1}]$ so that \be{m:1} \W_{w, \lv} = \sum_{\mv \in \Lv} \Upsilon_{w, \mv}(q^{-1}) e^{\mv}. \ee 

\begin{nlem} \label{Phi-wd} For $\mv \in \Lv,$ define $\Upsilon_{\mv}(v^{-1}) := \sum_{w \in W} \Upsilon_{w, \mv}(v^{-1}).$ Then we have $\Upsilon_{\mv}(v^{-1}) \in \C[ v^{-1}]$. \end{nlem} 
\begin{proof} From Proposition \ref{sym:wd} we know that $\mc{P} =  \sum_{\tau \in W} C_{\tau} [\tau] $ with $C_{\tau} \in \C[v^{-1}][[Q^{\vee}_-]].$ Then \be{P:lv} \mc{P}(e^{\lv}) = \sum_{\tau \in W} C_{\tau} e^{\tau \lv}. \ee For a fixed $\mv \in \Lv,$ we may argue as in \cite[\S 7.3.11]{bkp} to conclude that there are only finitely many $\tau \in W$ such that \be{Phi-1} [e^{\mv}]C_{\tau} e^{\tau \lv} \neq 0. \ee Since each $C_{\tau} \in \C[ v^{-1} ][[Q^{\vee}_-]]$ we conclude that \be{Phi-2} [e^{\mv}] \mc{P}(e^{\lv}) \in \C[v^{-1}]. \ee Hence, by definition and (\ref{W:t}), we conclude that $[e^{\mv}] q^{ \la \rho, \lv \ra} \mc{P}(e^{\lv}) = \Upsilon_{\mv}(v^{-1}),$ and so the Lemma follows. \end{proof}

\spoint \label{det-ct} From Proposition \ref{alg-iden} and the Weyl-Kac character formula (\ref{wk-fla}) we know that \be{P:d} \mc{P}(e^{\lv}) = \mf{c}(v) \;  \mf{D}_v \; \chi_{\lv} \ee  as expressions in $\C[v^{-1}]_{\leq}[\Lv],$ where $\mf{c}(v)$ is some $W$-invariant factor. Since the evaluation of $v^{ \la \rho, \lv \ra } \mc{P}(e^{\lv})$ at $v=q$ is equal to $\W(\pi^{\lv})$ we conclude that $\W(\pi^{\lv})$ is also the evaluation of the expression $ v^{ \la \rho, \lv \ra } \mf{c}(v) \;  \mf{D}_v \; \chi_{\lv}$ at $v= q.$ Let $\mf{c}$ be the evaluation of $\mf{c}(v)$ at $v=q.$ 

It remains to see that $\mf{c} = \mf{m},$ for which we pass to the Gindikin-Karpelevic limit as follows.  From (\ref{wk-fla}), we may write (using that $\Delta = \mf{D}_q / \mf{D}$) \be{wgk-1} e^{- \lv} \mf{c} \;  \mf{D}_q \; \chi_{\lv}  = \mf{c} \; \Delta \sum_{w \in W} (-1)^{\ell(w)} e^{w (\lv + \rho) - (\lv +  \rho) }. \ee Fix $\mv \in \Lv,$ and choose $\lv$ sufficiently dominant so that \be{wgk-2} [e^{\mv}]  \mf{c} \;  \mf{D}_v \; \chi_{\lv} = [e^{\mv}]  \mf{c} \; \Delta \sum_{w \in W} (-1)^{\ell(w)} e^{w (\lv + \rho) - (\lv +  \rho) } = [ e^{\mv} ] \mf{c} \; \Delta. \ee On the other hand, from Proposition \ref{gk-whit} we know that (perhaps increasing $\lv$) \be{wgk-3}  [e^{\mv - \lv}] \mf{c} \;  \mf{D}_v \; \chi_{\lv} = [e^{\mv - \lv}]  q^{ - \la \rho, \lv \ra}  \W(\pi^{\lv})  = [e^{\mv} ] \mf{m} \Delta. \ee As this holds for any $\mv$ we conclude that $\mf{c} \; \Delta = \mf{m} \; \Delta,$ and hence $\mf{c} = \mf{m}.$  

\appendix 
\section{Alternate proof of Whittaker Recursion } \label{app-hc}

The aim of this appendix is to sketch a proof of the recursion relation (\ref{whit:recur-body-1}) in the finite-dimensional case. We shall use an alternate integral presentation of the Iwahori-Whittaker function (see Proposition \ref{hc-fla}). The proof of the recursion using this formula is then carried out in \S \ref{rfs1}- \ref{rfs3}, following the technique of Macdonald \cite[Chapter 4]{mac:mad}. It would be interesting to see if Proposition \ref{hc-fla} can be generalized to the affine setting; the techniques of \S \ref{rfs1}-\ref{rfs3} carry over to the affine (or really any Kac-Moody) setting.

\subsection*{An Integral Formula} \label{int-fla} \noindent

\spoint For $w \in W_o$ we let $S_{w} = \{ a \in R_{o, +} \mid w^{-1} a < 0 \}.$ This is a finite set with the following recursive description: suppose that $w =  w'w_a$ with $\ell(w) = \ell(w')+1$ with $a \in \Pi_o.$ Setting $\gamma = w'(a)$, we  have $S_w = \{ \gamma \}  \sqcup S_{w'}.$ Let $U_w$ be the product of the root groups $U_a$ with $a \in R_w$ with coefficients in $\mc{K}$: in fact, $U_w$ admits a direct product decomposition \be{Uw:prod} U_w := \prod_{a \in S_w} U_a. \ee We define $U_{w, \O}$ and $U_{w, \pi}$ as the subgroups with coefficients (with respect to this direct product decomposition) that lie in $\O$ or $\pi \O.$   The set $U_{w}$ carries a natural Haar measure $du_w$ with respect to which $U_{w, \O}$ has volume $1.$ In the case $w=w_a$ is a simple reflection, we shall write $du_a$ in place of $du_{w_a}.$ Let \be{dense}  U^{\gen}_{w } := \{ u \in U_{w } \mid u w \in \um_o B_o  \}, \ee where by abuse of notation we write $w$ to represent a lift of an element of $W$ to $G.$  It is a subset of $U_{w}$ whose complement has measure $0.$ 

\spoint \label{a2} For $x \in \um_o B_o ,$ we may write (uniquely) $ x = \nm \pi^{\mv} h_{\O} n$ where $n \in \up_o, \nm \in \um_o, h_{\O} \in A_{o, \O}$ and $\mv \in \Lv.$ Using this decomposition, we may define $\niw: \um_oB_o \rr \um_o$ and $\aiw: \um_o B_o  \rr A'$ via the formulas, \be{n,a} \begin{array}{lcr} \niw(x) = \nm & \text{ and } & \aiw(x) = \pi^{\mv}. \end{array} \ee Thus we also obtain maps $\niw: U^{\gen}_w \, w \rr \um_o$ and $\aiw: U^{\gen}_w \, w \rr A'.$ 

Let $w \in W$ and recall that $\umw_o$ was defined in (\ref{Uo:w}) as the set of elements in $\um_o$ whose Iwasawa $K_o$-component lies in $I_o^- w I_o.$ For any $\xv \in \Lv_o$ let us set  \be{umvx} \umw_o(\xv):= \{ \nm \in \umw_o \mid \nm \in I_o^- w I_o \pi^{\xv} \up_o \}. \ee This set is invariant under left multiplication by $\um_{o, \O}$, and we set \be{M:x} \me{\xv}:=\um_{o, \O} \setminus \umw_o(\xv). \ee For $w \in W_o$ and $\mv \in \Lv_o$ we may also define \be{uw:mu} U_{w, \pi}(\mv):= \{ u_w \in U_{w, \pi} \mid u_w  w \in \um_o \pi^{\mv} A_{\O} \up_o \}. \ee Note that $U_{w, \pi}(\mv) \subset U_w^{\gen}.$

\begin{nlem} \label{change} Let $w \in W$ and $\mv \in \Lv.$ Then the map $\niw$ induces an injection $\niw: U_{w, \pi}(\mv) \, w \hookrightarrow \umw_o(-\mv),$ and the composition \be{ch:1} U_{w, \pi}(\mv) \,w  \rr \umw_o(-\mv) \rr \me{-\mv} \ee is surjective. \end{nlem} 

\rpoint{Remark} We denote the composition again as $\niw: U_{w, \pi}(\mv) \, w  \rr \me{-\mv}.$

We suppress the proof of this Lemma, and only note here that it follows from certain Iwahori-Matsumoto type decompositions: first, we have \be{im} I_o &=& U_{o,\O} \um_{o, \pi} A_{o, \O} =  \um_{o, \pi} U_{o,\O}  A_{o, \O} \\ I^-_o &=& \um_{o,\O} U_{o, \pi} A_{o, \O} =  U_{o, \pi} \um_{o,\O}  A_{o, \O} \ee (and note that $A_{o, \O}$ normalizes each of the other groups appearing in these decompositions); and second, the (opposite) Bruhat decomposition of $\mathbf{G}(\kk)$ may be lifted to yield the decomposition as in (\ref{Kim})  \be{K:op} K_o = \sqcup_{w \in W_o} I_o^- \, w \, I_o \ee where again by abuse of notation $w \in K_o$ denotes some chosen lift of the corresponding element in the Weyl group $W_o.$

\spoint The following integral formula will be used to prove the recursion result for $\W_{o, w}(\pi^{\lv}).$ 

\newcommand{\Phit}{\Phi^-}

\begin{nprop} \label{hc-fla} There exists a constant $C=C(G_o, \kk)$ depending only on $G_o$ and the residue field $\kk$ such that for every $w \in W_o$ and  $\lv \in \Lv_{o, +},$ the Iwahori-Whittaker function $\W_{o, w}(\pi^{\lv})$ introduced in (\ref{W:wpart}) is equal to the constant $C$ times the following integral expression, 
\be{Jw}  J_w(\pi^{\lv}) = \Phi_{-\rho_o}(e^{\lv}) \int_{U_{w, \pi}^{\gen} } \Phit_{\rho_o} (\aiw(u_w w)) \psi(\pi^{-\lv} \niw( u_w w) \pi^{\lv}) du_w, \ee where $U_{w, \pi}^{\gen} = U_{w, \pi} \cap U_{w}^{\gen},$ $du_w$ is the Haar measures as in \S \ref{a2}, and for any $\nu \in \Lv_o$ we define \be{phit} \Phit_{\nu}: A' \rr \C[\Lv_o], \ \  \pi^{\lv} \mapsto e^{- \lv} q^{ \la \nu, \lv \ra} ,\, \lv \in \Lv. \ee \end{nprop}

\rpoint{Remark} Note that since $U_{w, \pi}^{\gen} \subset U_{w, \pi}$ is dense. Hence if we extend the integrand in (\ref{Jw}) to a function on $U_{w, \pi}$ which is equal to zero for all $U_{w, \pi} \setminus U_{w, \pi}^{\gen},$ then we may write \be{Jw:ext} J_w(\pi^{\lv}) = \Phi_{-\rho_o}(e^{\lv}) \int_{U_{w, \pi} } \Phit_{\rho_o} (\aiw(u_w w)) \psi(\pi^{-\lv} \niw( u_w w) \pi^{\lv}) du_w. \ee  We shall keep this convention throughout the next proof-- i.e., our notation will not distinguish between $U_w$ and $U_w^{\gen},$ with the understanding that all quantities which are not defined on $U_w^{\gen}$ are assumed to be equal to $0.$

\begin{proof} For $\lv \in \Lv_{o,+}$, let us write $\psi_{\lv}(x) := \psi(\pi^{-\lv} x \pi^{\lv})$ for $x \in \um.$ By definition, \be{w:umw} \W_{o, w}(\pi^{\lv}) = \Phi_{- \rho_o}(e^{\lv}) \sum_{\xv \in \Lv} e^{\xv} q^{ \la \rho_o, \xv \ra} \sum_{x \in \me{\xv} } \psi_{\lv}(x). \ee By the way we have defined the measure on $U_o^-$ the right hand side of the above may be rewritten as \be{a3-1} \Phi_{- \rho_o}(e^{\lv}) \sum_{\xv \in \Lv} e^{\xv} q^{ \la \rho_o, \xv \ra} \int_{\umw_o(\xv)} \psi_{\lv}(\nm) d\nm. \ee Therefore, by comparing coefficients, it suffices to show that for some constant as in the statement, \be{comp:coe}  \int_{\umw_o(-\xv)} \psi_{\lv}(\nm) d\nm = C \, q^{ \la 2 \rho_o, \xv \ra} \int_{U_{w, \pi}(\xv)} \psi_{\lv}(\niw(u_w w)) du_{w}.\ee This may be deduced from a formula of Harish-Chandra which asserts the following: let $f: K_o/ K_o \cap B_o \rr \C$ be some function, then up to some constant $C'=C'(G_o, \kk) \in \C$ we have (see (\ref{iwa:K+}) for the notation) \be{int-khc} C'  \int_{\um_o} f(\iw_{K_o}(\nm)) \iw_{A_o}(\nm)^{2 \rho} d\nm =  \int_{K_o/ K_o \cap B_o} f(k) dk  \ee where $d \nm$ is the Haar measure on $U_o^-$ (we need not specify the normalization since we can absorb it into the constant $C'$), $dk$ is a quotient of the Haar measure on $K_o$ assigning total volume $1$ to $K_o$ and if $\iw_{A_o}(\nm) = e^{\mv}, \mv \in \Lv_o$ then $\iw_{A_o}(\nm)^{2 \rho}:= q^{ \la 2 \rho, \mv \ra}.$ 

Restricting the integral on the left hand side of (\ref{int-khc}) to the subset $\umw_o(-\xv),$ the integral on the right must be considered over the image of $\umw_o(-\xv)$ under the map $\iw_{K_o}.$ From Lemma \ref{change}, this image (modulo $K_o \cap B_o$) is equal to $\um_{o, \O} U_{w, \pi}(\xv)w.$ Assume further that $f$ is invariant under $\um_{o, \O}$ on the left, we have \be{int-khc:1} \int_{\um_{o, \O} U_{w, \pi}(\xv) w} f(k) dk= C'' \int_{U_{w, \pi}(\xv)} f(u_w w) du_w, \ee for some $C''=C''(G_o, \kk)$ with $du_w$ defined as above. Hence, for $f \in \um_{o, \O} \setminus K_o / K_o \cap B_o,$ there exists a constant $C'''=C'''(G_o, \kk)$ as above so that  \be{int-khc:2} C'''  \int_{\umw_o(-\xv)} f(\iw_{K_o}(\nm)) \iw_{A_o}(\nm)^{2 \rho} d\nm 
= \int_{U_{w, \pi}(\xv)} f(u_w w) du_w,  \ee and hence \be{int-khc:3} \int_{U_{w, \pi}(\xv)} f(u_w w) du_w &=&  C''' q^{ - \la 2 \rho, \xv \ra} \int_{\umw_o(-\xv)} f(\iw_{K_o}(\nm)) d\nm. \ee The proposition follows from this by choosing $f$ appropriately.

\end{proof}

\subsection*{Proof of Recursion}

\noindent We now sketch the proof of Proposition \ref{whit:recur-body-1}. In fact we need only verify that $J_w(\pi^{\lv})$ (and more precisely, its variant (\ref{Jw:ext})) verifies the same recursion in light of Proposition \ref{hc-fla}.

\tpoint{Step 1: Preliminary Reductions} \label{rfs1} Since we have a reduced decomposition $w=w'w_a$ (from now on, we abuse notation and use the same symbol to refer to elements of the Weyl group and their lifts to $K_o$) we have an equality of sets  \be{wUw} U_{w, \pi} w =   U_{w', \pi} w' U_{a, \pi} w_a. \ee Keeping the notation introduced above that $\psi_{\lv}(x) := \psi(\pi^{-\lv} x \pi^{\lv})$ for $x \in \um, \lv \in \Lv_{o, +},$ we may write   \be{fubini} J_w(\pi^{\lv}) = \Phi_{- \rho_o}(e^{\lv})\int_{U_{a, \pi}} \int_{U_{w', \pi}} \Phit_{\rho_o}( \aiw( u_{w'} w' u_a w_a) )\psi_{\lv}(\niw( u_{w'} w' u_a w_a) )  du_{w'} du_{a}, \ee where $du_{w'}$ and $du_a$ are the Haar measures on $U_{w'}$ and $U_a$ normalized as in \S \ref{a2}. Let us now introduce the non-compact set (see (\ref{Ua:defs}) for notation), \be{U:neg} U_a(\leq 0) = \sqcup_{k \leq 0} U_{a}[k]. \ee Since we have a decomposition of sets \be{sets} U_{a, \pi} = U_a \setminus U_a(\leq 0) \ee we have that \be{W:diffparts} J_{w}(\pi^{\lv}) = J_w^1(\pi^{\lv}) - J_w^2(\pi^{\lv}) \ee where \be{W-1} J_w^1(\pi^{\lv}) &=& \Phi_{- \rho_o}(e^{\lv})\int_{U_{a}} \int_{U_{w', \pi}} \Phit_{\rho_o}( \aiw( u_{w'} w' u_a w_a) \psi_{\lv}(\niw( u_{w'} w' u_a w_a) )  du_{w'} du_a \\
\label{W-2} J_w^2(\pi^{\lv}) &=& \Phi_{- \rho_o}(e^{\lv})\int_{U_a(\leq 0)} \int_{U_{w', \pi}} \Phit_{\rho_o}( \aiw( u_{w'} w' u_a w_a) \psi_{\lv}(\niw( u_{w'} w' u_a w_a ) )  du_{w'} du_a \ee It suffices to show the following two statements, \be{j1=} J_w^1(\pi^{\lv}) &=& \frac{1 - q^{-1} e^{-\av}}{1 - e^{\av}} (J_{w'}(\pi^{\lv}))^{w_a}  \\ \label{j2=} J_w^2(\pi^{\lv}) &=&  \frac{1 - q^{-1} }{1 - e^{\av}} J_{w'}(\pi^{\lv}). \ee 

\tpoint{ Step 2: Computing $J_w^1$ } \label{rfs2} We may decompose $U_o = U_a U_o^a$ where \be{U^a} U_o^a = \{ u \in U_o \mid w_a u w_a^{-1} \in U_o \}. \ee Moreover, $U_a$ normalizes $U_o^a$ as does the root group $U_{-a}.$ Now, let $u_{w'} \in U_{w', \pi},$ and write if possible (uniquely) \be{j1:1}  u_{w'} w' = n^- h n, \text{ where } n^{-} \in U_o^{\-}, h \in A_o, \text{ and } n \in \up_o. \ee Let $u_a \in U_a$ and consider the expression \be{w:1-a} u_{w'} w' u_a w_a = n^- h n u_a w_a  =  n^- h n^a n_a   u_a w_a  \ee where $n = n^a n_a $ with $n^a \in U_o^a$ and $n_a \in U_a.$ We leave it as an exercise to verify the following

\begin{nclaim} Let $\wt{n}_a:= h n'_a h^{-1}$ with $n'_a := n_a u_a.$ Then we have \be{j1:a} \aiw( u_{w'} w' u_a w_a) &=& \aiw(u_{w'} w' )^{w_a} \aiw(\wt{n}_a w_a) \\ \label{j1:n} \niw(u_{w'} w' u_a w_a) &=& \niw(u_{w'} w') \niw(\wt{n}_a w_a). \ee  \end{nclaim} 

\noindent We may thus write $J_w^1$ (defined in \ref{W-1}) as \be{j1:2} \Phi_{- \rho_o}(e^{\lv}) \int_{U_{w', \pi}} \int_{U_{a}}\Phit_{\rho_o}( \aiw( u_{w'} w')^{w_a} \aiw( \wt{n}_a w_a) ) \psi_{\lv}(\niw( u_{w'} w' ) \niw( \wt{n}_a w_a) )  du_{w'} d\wt{n}_a \ee where $d\wt{n}_a$ is related to $du_a$ through a conjugation by $\aiw(u_{w'} w').$ Hence, \be{jac-ft} d \wt{n}_a = \mf{j}(u_{w'} w') du_a \ee where the Jacobian factor $\mf{j}(u_{w'} w')$ is equal to $q^{ \la a ,\mv \ra}$ if $\aiw(u_{w'} w') = e^{\mv}.$ Thus (\ref{j1:2}) may be factored as \be{j1:3} 
\overbrace{ \Phi_{- \rho_o}(e^{\lv}) \int_{U_{w', \pi}} \Phit_{\rho_o}( \aiw( u_{w'} w')^{w_a} ) \psi_{\lv}(\niw( u_{w'} w' )) \mf{j}(u_{w'} w') du_{w'}  }^{(I)} &\times & \\ \underbrace{ \int_{U_a} \Phit_{\rho_o} ( \aiw( u_a w_a) ) \psi_{\lv}(\niw( u_a w_a) )   du_a.}_{(II)}  \ee For any $a \in \Pi$ and $\mv \in \Lv_o$, we have \be{j1:4} \Phit_{\rho_o} (e^{\mv})^{w_a} = \Phit_{\rho_o} (e^{w_a \mv}) q^{ \la a, \mv  \ra }. \ee Thus we find that for the first term of (\ref{j1:3}), \be{j1:5} 
(I) &=& \Phi_{- \rho_o}(e^{\lv}) \int_{U_{w', \pi}} \Phit_{\rho_o}( \aiw( u_{w'} w') )^{w_a} \psi_{\lv}(\niw( u_{w'} w' ))  du_{w'} 
\\ &=& e^{  \la a, \lv \ra a^{\vee} } J_{w'}(\pi^{\lv})^{w_a} \ee A rank one computation shows that \be{rk1} \int_{U_{a} }  \Phit_{\rho_o}(\aiw(w_a u_a) ) \psi_{\lv} ( \niw( w_a n_a )  ) du_a  = e^{- \la a, \lv  \ra a^{\vee} }\frac{1 - q^{-1} e^{-\av}}{1 - e^{\av}}. \ee The statement (\ref{j1=}) now follows.

\tpoint{Step 3: Computing $J_w^2$ } \label{rfs3} Now we turn to $J_w^2(\pi^{\lv}),$ and use the decomposition $U_{a}(\leq 0) = \sqcup_{\ell \geq 0} U_a[-\ell]$ to write $J_w^2(\pi^{\lv}) $ (defined in \ref{W-2}) as \be{w2-1}  \sum_{\ell \geq 0} \int_{U_a[-\ell]} \int_{U_{w', \pi}}\Phit_{\rho_o} ( \aiw(u_{w'}w' u_{a}w_a )  \psi_{\lv}(\niw( u_{w'}w' u_{a} w_a) ) du_{w'} du_a. \ee   If $u_a \in U_a[-\ell]$ with $\ell\geq 0$ then we may write \be{rk1-uak}  u_a w_a = u_{-a}(\ell) h(\ell) u_{a}(\ell) \text{ where } u_{-a}(\ell) \in U_{a}[\ell],\,  h(\ell) \in \pi^{-\ell \av} A_{o, \O}, \, \text{ and } u_a(\ell) \in U_a[\ell]. \ee Thus, (\ref{w2-1}) equals $\Phi_{- \rho_o}(e^{\lv})$ times the inetgral
\be{j2-1} \sum_{\ell \geq 0} \int_{U_a[-\ell]} \int_{U_{w', \pi}}\Phit_{\rho_o} ( \aiw(u_{w'}w' u_{-a}(\ell) h(\ell) u_a(\ell) )  \psi_{\lv}(\niw( u_{w'}w' u_{-a}(\ell) h(\ell) u_a(\ell)) ) du_{w'} du_a.\ee 

\noindent For $u_{-a}(\ell)$ as above, let $\gamma = w'(a)$ and set $u_{-\gamma}(\ell) := w' u_{-a}(\ell) w'^{-1}.$  For $u_{w'} \in U_{w', \pi}$ we may now consider the conjugate, $x:= u_{- \gamma}(\ell)^{-1} u_{w', \pi} u_{- \gamma}(\ell).$ It is easy to verify that  $ x \in w' \um_o w'^{-1} \cap K_o.$ For any $x \in w' \um_o w'^{-1},$ we may write \be{x} x = x_- x_+ \ee where $x_- \in w' \um_o w'^{-1} \cap \um_o$ and $x_+ \in w' \um_o w'^{-1} \cap \up_o.$ Now, consider the map \be{eta:0} \eta_{u_{a}(\ell)}: U_{w', \pi} \rr U_{w', \pi} ,\,   u_{w'} \mapsto ( u_{- \gamma}(\ell)^{-1} u_{w'} u_{- \gamma}(\ell))_+. \ee We may write (\ref{j2-1}) as \be{j2-2}  \sum_{\ell \geq 0} \int_{U_a[-\ell]} \int_{U_{w', \pi}} \Phit_{\rho_o} ( \aiw( \eta_{u_{a}(\ell)} (u_{w'}) w'  h(\ell)  )  \psi_{\lv}(\niw( \eta_{u_{a}(\ell)} (u_{w'})  w' ) ) du_{w'} du_a, \ee where we have used the fact that  
\be{n-fact} \ u_{- \gamma}(\ell)^{-1} u_{w', \pi} u_{- \gamma}(\ell) \in \um_o(\O)  \eta_{u_{a}(\ell)} (u_{w'})   \ee and that, since $\lv \in \Lv_{o, +},$ $\psi_{\lv}$ is trival on $\um_{o, \O}$. Moreover, since the map $\eta_{u_{a}(\ell)}$ is bijective and measure preserving we see that (\ref{j2-2}) factors as \be{j2-3}  \int_{U_{w', \pi}} \Phit_{\rho_o} ( \aiw( u_{w'} w' )  \psi_{\lv}(\niw( u_{w'}w' ) ) du_{w'} \, \times \, \int_{U_a(\leq 0)}  \Phit_{\rho_o} ( \aiw( w_a u_a ) ) du_a.
\ee The proof of (\ref{j2=}) is concluded from the rank $1$-computation, \be{rk1-2}  \int_{U_a(\leq 0)}  \Phit_{\rho_o} ( \aiw( w_a u_a ) ) du_a = \frac{1 - q^{-1} }{1 - e^{\av}} \ee where the right hand side is expanded as a rational function in positive powers of $e^{\av}.$


\end{document}